\documentclass[11pt,reqno]{amsart}
\usepackage{color}
\usepackage{mathtools}
\usepackage[sort,numbers]{natbib}
\usepackage{bbm}

\usepackage{tikz}
\usetikzlibrary{shapes.misc}
\usepackage{graphicx}
\usepackage{subcaption}
\usepackage{tikz}
\usetikzlibrary{decorations.markings}
\usepackage{soul}
\tikzset{
	cross/.style={path picture={ 
			\draw[black] (path picture bounding box.south west) -- (path picture bounding box.north east); 
			\draw[black] (path picture bounding box.south east) -- (path picture bounding box.north west);
	}},
	otimes/.style={circle, draw, inner sep=0pt, minimum size=5pt, path picture={
			\draw[black] (path picture bounding box.south west) -- (path picture bounding box.north east); 
			\draw[black] (path picture bounding box.south east) -- (path picture bounding box.north west);
		}, fill=white}
}

\usepackage{hyperref}
\hypersetup{
	colorlinks,
	linkcolor=blue,
	citecolor=blue,
	filecolor=blue,
	urlcolor=blue,
	pdftitle={},
	pdfsubject={},
	pdfkeywords={}
}

\usepackage[utf8]{inputenc} 
\usepackage[T1]{fontenc}    
\usepackage{url}            
\usepackage{booktabs}       
\usepackage{amsfonts}       
\usepackage{nicefrac}       
\usepackage{microtype}      
\usepackage{amsmath} 
\usepackage{amsthm}
\usepackage{amssymb}
\usepackage{mathrsfs}
\newtheorem{theorem}{Theorem}[section]

\newtheorem{lemma}[theorem]{Lemma}

\newtheorem{corollary}[theorem]{Corollary}

\newtheorem{question}[theorem]{Open Question}
\theoremstyle{definition}
\newtheorem{definition}[theorem]{Definition}

\theoremstyle{remark}

\usepackage{comment}
\allowdisplaybreaks
\usepackage[titletoc]{appendix}
\usepackage{systeme}

\newcommand{\pr}{\mathbb{P}}

\newcommand{\R}{\mathbb{R} }
\newcommand{\Z}{\mathbb{Z}}

\usepackage{tikz}
\usepackage{pgfplots}
\usepgfplotslibrary{external}
\pgfplotsset{compat=1.10}
\usetikzlibrary{positioning}
\usetikzlibrary{patterns}
\usetikzlibrary{intersections,arrows.meta,positioning,calc}
\usepackage{comment}
\usepackage[shortlabels]{enumitem}
\usepackage{todonotes}

\usepackage{caption}
\usepackage[margin=1in,footskip=0.25in]{geometry}

\tikzset{cross/.style={cross out, draw=black, minimum size=2.5*(#1-\pgflinewidth), inner sep=2pt, outer sep=0.5pt},
	cross/.default={1pt}}
\usepackage{xcolor}
\usetikzlibrary{calc,arrows}
\newcommand{\boundellipse}[3]
{(#1) ellipse (#2 and #3)
}

\newcommand{\E}{\ensuremath{\mathbb{E}}}
\newcommand{\N}{\ensuremath{\mathbb{N}}}

\newcommand{\sign}[1]{\mathrm{Sign} }
\newcommand{\la}{\lambda}
\newcommand{\wla}{\widehat{\lambda}}
\newcommand{\wal}{\widehat{\alpha}}
\newcommand{\T}{\ensuremath{\mathbb{T}}}

\usepackage{environ}
\NewEnviron{eq}{%
	\begin{equation}\begin{split}
			\BODY
	\end{split}\end{equation}
}

\newcommand{\Exp}{\text{Exp}}

\newcommand{\PPP}{\text{PPP}}

\numberwithin{equation}{section}

\usepackage{algorithm}
\usepackage{algpseudocode}



\renewcommand{\P}{\mathbb P}
\renewcommand{\E}{\mathbb E}
\newcommand{\ep}{\varepsilon}

\title{Supercriticality of epidemics on Random Networks}
\author{Phuc Lam and Oanh Nguyen}
\address{Division of Applied Mathematics\\ Brown University\\  Providence, RI 02906, USA}
\email{phuc\_lam@brown.edu}
\address{Division of Applied Mathematics\\ Brown University\\  Providence, RI 02906, USA}
\email{oanh\_nguyen1@brown.edu}
\date{}
\thanks{Nguyen is supported by NSF grant DMS-2246575.}

\let\oldtocsection=\tocsection

\let\oldtocsubsection=\tocsubsection

\let\oldtocsubsubsection=\tocsubsubsection

\renewcommand{\tocsection}[2]{\hspace{0em}\oldtocsection{#1}{#2}}
\renewcommand{\tocsubsection}[2]{\hspace{2em}\oldtocsubsection{#1}{#2}}
\renewcommand{\tocsubsubsection}[2]{\hspace{4em}\oldtocsubsubsection{#1}{#2}}

\begin{document}	
	\maketitle
	\begin{abstract} We analyze the supercritical regime of the SIRS process (Susceptible-Infected-Recovered-Susceptible) on random graphs and Galton-Watson trees.   Our results provide a three-part characterization of this process:  first, we prove that even on finite sparse graphs, the epidemic can survive for an exponentially long time. On heavy-tailed networks with power-law-like exponents, we show that this holds for the entire range of parameters. Notably, while contemporary works establish a stretched exponential survival time for power-law networks, our approach provides the optimal bound of order $e^{\Theta(n)}$ via a different approach. Finally, for infinite trees, we establish sufficient conditions for strong survival, proving that the root is re-infected infinitely often even for light-tailed degree distributions like the Poisson distribution.
	\end{abstract}
	
	\section{Introduction}
	
	The SIRS (Susceptible-Infected-Recovered-Susceptible) process models the spread of epidemics on networks. For a given graph $G = (V, E)$, the SIRS process with infection rate $\la$, recovery rate $1$, and deimmunization rate $\alpha$ is defined as a continuous-time Markov chain on the state space $\{S, I, R\}^V$. The dynamics are governed by the following independent transitions:
	\begin{itemize}
		\item Infected vertices transmit the infection to each of its neighbors at Poisson rate $\la$;
		\item Infected vertices transition to the recovered state at Poisson rate $1$;
		\item Recovered vertices remain immune to further infection, then returning to the susceptible state at Poisson rate $\alpha$;
		\item {Independence:} All infections, recovery, and deimmunization events occur independently.
	\end{itemize}
	
	By incorporating a temporary period of immunity, the SIRS model effectively captures the behavior of many diseases such as influenza and COVID-19. It has been widely applied to study epidemic spreading and forest fire dynamics by a wealth of empirical literature \cite{kuperman2001small, mollison1986modelling, mollison1985spatial, wang2017spreading}. However, the mathematical analysis of the SIRS process is notoriously challenging due to its lack of monotonic properties. Consequently, rigorous theoretical results are sparse, including \cite{durrett1991epidemics, sivakoff2022, friedrich2022analysis, lam2024optimal, souvik2026sirs, gobel2026reemergence}. A more comprehensive overview of these challenges is provided in Chapter 4 of Durrett's recent monograph \cite{durrett2021dynamics}.
	
	A more tractable variant of this model is the {SIS process}, also known as the contact process, which emerges as $\alpha \to \infty$ (instantaneous loss of immunity). In contrast to the general SIRS model, the SIS model enjoys various nice properties, including attractiveness and monotonicity. Since its introduction by Harris in 1974 \cite{harris1974}, the survival and phase transition properties of the SIS process have been heavily investigated on different networks. On infinite structures, phase-type results for the infection rate $\lambda$ have been established for the integer lattice $\mathbb{Z}^d$ \cite{harris1974}, $d$-regular trees \cite{liggett96, pemantle92, stacey96}, and Galton-Watson trees \cite{bhamidi2021, nam2022critical, nguyen2022subcritical, huangdurrett20}. On finite structures, results on the survival time of the SIS process have been established on lattice cubes \cite{durrettliu88, durrettschon88, mountford93}, trees of finite depth \cite{cranston2014, stacey01}, random $d$-regular graphs \cite{lalleysu17, mourrat16}, and sparse graphs \cite{ChatterjeeDurrett2009, schapira2017extinction, bhamidi2021, nam2022critical, nguyen2022subcritical}; the list is far from complete. We point interested readers to the book \cite[Chapter 6]{liggett:ips} and the survey \cite{valesin2024contact} for a more detailed introduction to this model.
	
	Phase transitions of SIS model, namely the existence of a critical infection rate below which the epidemic dies out globally on infinite networks, are well studied. For example, Huang--Durrett \cite{huangdurrett20} and Bhamidi--Nam--Nguyen--Sly \cite{bhamidi2021} established a fundamental characterization of the contact process by proving that the extinction--survival threshold $\lambda_{1}$ for a Galton-Watson tree is strictly positive if and only if the offspring distribution has a finite exponential tail. Extending these results to finite random graphs with degree distribution $\mu$ having an exponential tail, they demonstrated a phase transition: for sufficiently small infection rates, the process exhibits short survival lasting only $n^{1+o(1)}$ time, whereas for large enough rates, it achieves long survival persisting for $e^{\Theta(n)}$ time.

	For the SIRS model, these questions remain wide open. Due to the lack of monotonicity and attractiveness, it is not even known if there is a phase transition at all, for instance. In this paper, we investigate strong survival and long-term persistence of the SIRS process on both infinite trees and random graphs generated via the configuration model. Our contribution is threefold and is discussed in the next three subsections.

	\subsection{Exponential survival in the sparse random graphs} Random graphs are important models for understanding how real-world networks stay connected, especially in sparse cases where vertices typically have low degree. In this paper, we consider two popular models of random graphs with given degree distribution $\mu$: the random graph $\mathcal G(n, \mu)$ and the configuration model $\mathcal G_{conf}(n, \mu)$. In Section \ref{subsec: rg basic}, we provide a detailed description and some useful known results about these models. Prior work by Friedrich--Göbel--Klodt--Krejca--Pappik \cite{friedrich2022analysis} has established that SIRS persists for exponential time on dense graphs (namely \cite[Corollary 1.5]{friedrich2022analysis} for Erdos-Renyi random graphs with average degree $\gg \log n$). Here, we show that even sparse graphs, those with bounded average degree, possess the necessary structural properties to sustain the process. Specifically, the following theorem demonstrates that for SIRS process on a random graph $G_n$ with $n$ vertices and degree distribution $\mu$ under the standard supercriticality condition $\mathbb{E}[D(D-2)] > 0$ where $D\sim \mu$, the infection rate $\lambda$ has a long-survival phase where the infection persists for time $e^{\Theta(n)}$, provided the deimmunization rate $\alpha$ is sufficiently large.

	\begin{theorem}\label{thm:longSurvivalAlmostSIS}
		Suppose $\mu \in \mathcal P(\N)$ satisfies $$\E D(D - 2) > 0,$$
		where $D\sim \mu$. Consider the SIRS process on $G_n \sim \mathcal G_{conf}(n, \mu)$, where all vertices are initially infected. Then there exists a non-negative constant $c_{\mu}$ such that for $\alpha > c_{\mu}$ and $\la > c_{\mu}$, the survival time of SIRS process on $G_n$ is $\exp(\Theta_{\la, \alpha, \mu}(n))$ w.h.p.
		
		If we assume furthermore that $\E D^2<\infty$, then the same holds for $G_n\sim \mathcal G(n, \mu)$.
	\end{theorem}

	\subsection{The Absence of a Subcritical Phase in Heavy-Tailed Networks} We examine the behavior of SIRS on heavy-tailed graphs, such as those following a power-law degree distribution with exponent in $(2, 3)$. In the next theorem, we prove that for such networks, there is no short survival phase. In other words, for {any} fixed rates $\alpha > 0$ and $\lambda > 0$, the infection survives for exponential time w.h.p. 
	
	\begin{theorem}\label{thm:noSubcritPhaseInHeavyTail}
		Suppose $\mu \in \mathcal P(\N)$ satisfies that for $D\sim \mu$,
		\begin{itemize}
			\item $\E D = d \in (1, \infty)$, 
			\item There exists a function $f: \N\rightarrow \R^+$ for which $$ \limsup_M M^2 \mu([M, f(M)]) = \infty. $$
		\end{itemize} 
		
		Consider the SIRS process on $G_n \sim \mathcal G_{conf}(n, \mu)$, where all vertices are initially infected. Then for all fixed $\alpha > 0$ and $\lambda > 0$, the survival time of the SIRS process on $G_n$ is $\exp(\Theta_{\lambda, \alpha, \mu}(n))$ w.h.p. In other words, we have $c_{\mu}=0$ in the setting of Theorem \ref{thm:longSurvivalAlmostSIS}.
	\end{theorem}
	Distributions satisfying the conditions of Theorem \ref{thm:noSubcritPhaseInHeavyTail} include power laws with exponents in $(2, 3)$ or even those following $\mu(k) \propto \frac{\log k}{k^3}$, where one may simply set $f(M) = 2M$. 
	
	With regard to Theorem \ref{thm:noSubcritPhaseInHeavyTail}, we would like to acknowledge the beautiful work of He, Dhara, and Mukherjee \cite{souvik2026sirs}, which was brought to our attention during the final days of this project, and of Chatterjee, Sivakoff, and Wascher \cite{chatterjee2025effects} shortly after. Our results were developed independently with a completely different method and offer a complementary perspective on the long survival phase of the SIRS process. \cite{souvik2026sirs} covers the full spectrum of power-law exponents greater than 2, while \cite{chatterjee2025effects} covers power-law exponents greater than 3. Our current proof, on the other hand, covers exponents less than 3. Compared to these papers, despite this narrower range (compared to \cite{souvik2026sirs}), we achieve the optimal bound, using a simpler proof: while both papers establish a stretched exponential survival time of $e^{n^{1-\varepsilon}}$, we provide the correct order of $e^{\Theta(n)}$. 
	
	\subsection{Strong Survival on Infinite Trees}
	Moving on to infinite trees, questions about supercriticality correspond to asking when the SIRS process survives forever with positive probability. There are two possibilities for survival, weak survival (the root is reinfected only finitely many times, i.e. global but not local survival) and strong survival (the root is infected infinitely often, i.e. local survival). Using a standard argument with super-critical branching process, for a Galton-Watson tree with mean offspring $d>1$, for any $\alpha>0$, if $\wla d>1$, there are infinitely many vertices infected even in the first sweep of the process, since $\wla d$ is the expected number of susceptible children of a vertex getting the infection before the parent recovers. So we always have survival with positive probability for sufficiently large $\lambda$. Then, the question of interest would be to ask for strong survival. From the upper bound for survival of SIRS on star graphs \cite{friedrich2022analysis, lam2024optimal}, we see that unlike SIS, increasing the infection rate $\la$ does not help prolonging survival of SIRS indefinitely. Low-degree vertices are genuine bottlenecks for SIRS, so the possibility of strong survival remained unclear.
	
	As progress towards this question, we show that strong survival is indeed possible by establishing a lower bound for strong survival on trees with a degree distribution $\mu$. We show that if the deimmunization rate $\alpha$ and infection rate $\lambda$ satisfy a specific coupled inequality related to $\mu$, the infection can persist indefinitely at the root. Specifically, Theorem \ref{thm:mainThmOfStrongSurvival} provides sufficient conditions under which the root is re-infected infinitely often, even for light-tailed distributions such as Poisson. 
	
	For any parameter $p > 0$, denote $\widehat{p} := \frac{p}{p + 1} \in (0, 1)$. Our third theorem is as follows.
	\begin{theorem}\label{thm:mainThmOfStrongSurvival}
		Suppose $\mu \in \mathcal{P}(\N)$ has unbounded support. Moreover, let $d := \E_{D \sim \mu} D > 1$ and $\alpha$ be such that 
		\begin{enumerate}
			\item $\alpha > \frac{1}{\sqrt{d}-1}$ (i.e. $\wal > \frac{1}{\sqrt{d}}$),
			\item There exists $\delta>0$ for which
			$\limsup \mu([n, \infty)) \exp(n^{\alpha - \delta}) = \infty$.
		\end{enumerate} 
		Let $\lambda$ be such that   $\la > \frac{\alpha + 2}{\alpha + 1} \cdot \left( \wal \sqrt{d} - 1 \right)^{-1}$. Consider any deterministic (possibly infinite) tree $\T$, and at some vertex $\rho \in \T$, we append a Galton-Watson tree $\mathcal T \sim GW(\mu)$ rooted at $\rho$. We call this newly created tree $\T(\rho, \mu)$ and  consider the SIRS process there. Then there exists a constant $p = p(\mu, \lambda, \alpha) > 0$ such that starting with any initial configuration $\mathcal C_0$ with finitely many infections including $\rho$, we have 
		$$ \P_{\mathcal C_0}(\rho \text{ is re-infected i.o.}) \ge p. $$
        In particular, this holds for $\mu \sim \text{Pois}(d)$, $d > 1$, and $\alpha > \max\left(1,  \frac{1}{\sqrt{d}-1} \right)$.
	\end{theorem}
	
	Our result carries over to any hosting trees $\T$ containing the Galton-Watson tree, which is not straightforward without monotonicity. Moreover, the provided lower bound for strong survival does not depend on the hosting tree or the initial configuration, as long as the root $\rho$ is infected at the beginning.
	
	For contact processes, the dichotomy between survival of sparse graphs and that on their local-weak limit is a well-explored theme. For example, the supercritical phase on $d$-regular graphs coincides with the global survival phase on $d$-regular trees (see \cite{lalleysu17}, \cite{mourrat16}). More recently, it is shown in \cite{jahnel2025phasetrans} that the exponentially-long survival phase on sparse graphs is contained within the survival phase on their local-weak limit, and equality occurs in the configuration model vs. (size-biased) Galton-Watson tree settings. However, {exponentially-long survival phase on sparse graphs does not differentiate weak and strong survival phases on the local limit}, i.e. whether the root is re-infected infinitely often with some positive probability or not. On the other hand, it is straightforward that SIR processes on trees may have a weak survival phase, but never a strong survival phase. 
	
	Since contact processes and SIR processes can be viewed as limits of the SIRS model with deimmunization rates $\alpha = \infty$ and $\alpha = 0$ respectively, it is natural to ask whether a strong survival phase always exists for the SIRS process on trees. In both Theorems \ref{thm:longSurvivalAlmostSIS} and \ref{thm:mainThmOfStrongSurvival}, we note the constraint imposed on the deimmunization rate $\alpha$ by the measure $\mu$. It is probable that there is always a strong survival phase for the SIRS process whenever $\alpha > 0$, but proving it turns out to be very challenging. For small de-immunization rates, we state the following open question, which seems to be a reasonable starting point.
	\begin{question}
		For SIRS on $\Z$, for all sufficiently small $\alpha > 0$, does there exist $\lambda_0 > 0$ such that for all $\la > \la_0$, the process survives strongly? Or does there exist a sequence $(\la_k) \nearrow \infty$ for which the process becomes extinct for all $\la = \la_k$?
	\end{question}
	Towards this question, it has been proven in \cite{heeszel2026thesis} that similar to the SIS process, SIRS does not survive weakly on $\Z$. Moreover, it is not hard to prove the existence of such $\la_0$ when $\alpha$ is sufficiently large, using a standard block argument in \cite{bezuidenhout90}. 
	
	\subsection{Organization} After presenting some preliminaries in Section \ref{sec:prelims}, we will introduce the forced-infection process and collect some key facts in Section \ref{sec:forcedInfection}. Theorems \ref{thm:longSurvivalAlmostSIS}, \ref{thm:noSubcritPhaseInHeavyTail}, and  \ref{thm:mainThmOfStrongSurvival} are then proved in Sections \ref{sec:longSurvivalAlmostSIS},  \ref{sec:noSubcritPhaseInHeavyTail}, and \ref{sec:mainThmOfStrongSurvival}, respectively.
	
	\subsection{High-level proof sketch and our contribution}
	To overcome non-monotonicity issues of the SIRS, all of our proofs rely on the construction and analysis of the \textit{forced-infection process}, which is a monotone process whose infection set stochastically dominates that of the original SIRS process from below at all times. This process comes from the observation that if the last Poisson clock that rings at a site prior to an incoming infection is a de-immunization clock, then the incoming infection must be effective.
	
	The proofs of Theorems \ref{thm:longSurvivalAlmostSIS} and \ref{thm:noSubcritPhaseInHeavyTail} rely on a unified approach: we establish the deterministic existence of a robust, highly connected subgraph embedded within the random network that acts as a self-sustaining reservoir for the disease. This general strategy builds upon the embedded expander framework developed by Bhamidi, Nam, Nguyen, and Sly \cite{bhamidi2021} to prove long survival in the SIS model. However, the implementation of this approach diverges based on the network's tail behavior. For the sparse graphs in Theorem~\ref{thm:longSurvivalAlmostSIS}, the original embedded construction in \cite{bhamidi2021} suffices at the expense of high de-immunization rates.
	
	For the heavy-tailed graphs in Theorem~\ref{thm:noSubcritPhaseInHeavyTail}, we introduce a different construction to address the shorter local survival of the SIRS process (or more precisely, the forced-infection process). Unlike the earlier construction which relied on high-degree hubs as ``batteries'' to bridge multi-step distances in graphs with exponential tails, we establish a robust expander core characterized by direct one-step internal connectivity. We achieve this by coloring and matching a fixed number of ``blue'' edges via a tailored cut-off line algorithm, ensuring that high-degree neighbors are directly adjacent so that the local drift toward expansion overwhelmingly dominates individual recoveries.  
	
	Finally, to prove Theorem~\ref{thm:mainThmOfStrongSurvival}, we use a recursive approach from Pemantle \cite{pemantle92} and Huang -- Durrett~\cite{huangdurrett20}. We track the probability of the root being infected within a time frame with fixed length, using survival behavior of a high-degree star (i.e. the root and its many neighbors). In this step, the idea is to use estimates on the survival time and round structures of star graphs developed in our earlier work~\cite{lam2024optimal}. However, we actually need to get a lower bound of the survival time on the forced-infection process on stars rather than the original process, which turned out to deviate from the aforementioned previous work. By utilizing these star graph results within a recursive strategy, we prove that when $\la > \frac{\alpha + 2}{\alpha + 1} \cdot \left( \wal \sqrt{d} - 1 \right)^{-1}$, the upward drift from high-degree hubs at a specified level creates a self-sustaining cycle that continually pushes the disease back up to the root.
	
	\subsection{Notation}
	To ensure consistency of notation throughout this paper, we use the following.
	\begin{itemize}
		\item Unless otherwise mentioned, $(\mathcal F_t)_{t \ge 0}$ is the natural filtration associated with the random process on the graph in consideration. Moreover, on a random graph $G$, $\sigma(G) \subseteq \mathcal F_0$, i.e. all the information about the random graph is known at time $t = 0$.
		\item PPP means "Poisson point process". We denote $\frak{R} = (\frak{R}_t)_{t \ge 0} \sim \PPP(\beta)$ if $\frak{R}$ (Fraktur R) is a Poisson point process with rate $\beta > 0$. We will mostly use Fraktur font for PPPs.
		\item For any subgraph $H \subseteq G$, let $\mathcal F(H)$ be the $\sigma$-algebra generated by (i) the recovery and deimmunization clocks on $V(H)$ and (ii) the infection clocks along edges $E(H)$.
		
		\noindent We also use $\overline{\mathcal F}(H)$ to denote any arbitrary $\sigma$-algebra that is independent of $\mathcal F(H)$. Depending on our purposes, we can flexibly choose $\overline{\mathcal F}(H)$.
		\item For a given graph $G = (V, E)$ and $A \subseteq V$, $R \in \N$, $N(A, R)$ denotes the set of vertices in $V$ that is of distance at most $R$ from some vertex in $A$.
		\item $\text{Hypo}(\la_1, \dots, \la_m)$ denotes the hypoexponential distribution with parameters $\{\la_i\}_{1 \le i \le m}$, i.e. the sum of $m$ independent exponentially distributed random variables with rates $\la_1, \dots, \la_m$.
		\item For an event $A$, a $\sigma$-algebra $\mathcal F$, and random variable $X$, denote $\E(X \mid \mathcal F, A) := {\bf 1}_A \E(X \mid \mathcal F)$.
		\item If $X$ is a parameter of a random graph, then $X = o_{\pr}( g(n) )$ means that $X/ g(n)$ converges to $0$ in probability as $n \rightarrow \infty$. 
        \item For positive $f, g$, we write $f(n) \lesssim_{l_1, \dots, l_k} g(n)$ (resp. $f(n) \gtrsim_{l_1, \dots, l_k} g(n)$) if there exists a constant $C > 0$ depending on $l_1, \dots, l_k$ such that for all $n \ge 1$, $f(n) \le C g(n)$ (resp. $f(n) \ge Cg(n)$). Equivalently, we can also use the big-$O$ notation $f(n) = O_{l_1,\dots, l_k}(g(n))$ (resp. $f(n) = \Omega_{l_1,\dots, l_k}(g(n))$). 
        \item We either write $f \approx g$ or $f(n) = \Theta(g(n))$ if $f \lesssim g$ and $f \gtrsim g$.
		\item Unless stated otherwise, throughout this paper, we will write "w.h.p." to denote that an event is happening with high probability with respect to the \textit{joint} law of the random graph $G$ and the random process on it.
		\item For any parameter $p > 0$, we denote $\widehat{p} := \frac{p}{p + 1} \in (0, 1)$.
	\end{itemize}
	More notations are introduced during the course of our proofs.
	
	\subsection{Acknowledgements}
	AI tools (GPT and Claude) were used to compute the hypoexponential distribution in Lemma \ref{lm:probInfectFarWithinShort}, literature search, language assistance, and proofreading. All other ideas in the paper are entirely due to the authors. The final version of this paper is entirely human written, and the authors retain full responsibility for the paper's content. 
	
	\section{Preliminaries}\label{sec:prelims}
	\subsection{Random graphs}\label{subsec: rg basic}
	
	The random graph $\mathcal{G}(n,\mu)$ based on a probability distribution $\mu$ over $\mathbb{N}$ and an integer $n$ is defined as follows.
	\begin{itemize}
		\item Draw $n$ i.i.d. variables $d_1,\ldots , d_n$ from $\mu$, subject to the requirement that $\{\sum_{i=1}^n d_i \textnormal{ is even} \}$.
		
		\item Select a graph $G$ uniformly at random from the set of all simple graphs on $n$ vertices that satisfy the degree sequence $\{ d_i\}_{i=1}^n$.
	\end{itemize}
	
	\noindent The configuration model, denoted $\mathcal{G}_{conf}(n,\mu)$, serves as a variation of the aforementioned model:
	\begin{itemize}
		\item Generate the sequence $d_1,\ldots , d_n$ as described above, where each $d_i$ represents the count of \textit{half-edges} associated with vertex $i$.
		
		\item Form the graph by applying a uniform random pairing to all available half-edges.
	\end{itemize}
	Unlike $\mathcal{G}(n,\mu)$, the resulting graph $G$ in the configuration model may contain self-loops or multiple edges. However, provided that $\mu$ possesses a finite second moment, $\{\mathcal G(n, \mu)\}_n$ contiguous with respect to $\{\mathcal{G}_{conf}(n,\mu)\}_n$. For a comprehensive discussion, refer to \cite[Chapter 7]{vanderhofstad17}. 
	
	\begin{lemma}[\citep{j09, vanderhofstad17}]
		If $\E_{D\sim \mu} D^2<\infty$, then the probability $\P_{G\sim \mathcal{G}_{conf}(n,\mu)} (G \textit{ is simple})$ remains bounded within $(0,1)$ uniformly in $n$. 
		Consequently, for any sequence of graph properties $A_n$, 
		$$\P_{G\sim \mathcal{G}_{conf}(n,\mu)} (G\in A_n) \rightarrow 0 \quad\textnormal{yields}\quad \P_{G\sim \mathcal{G}(n,\mu)} (G\in A_n) \rightarrow 0.$$
		
		On the other hand, if $\E_{D\sim \mu} D^2=\infty$, then $\P_{G\sim \mathcal{G}_{conf}(n,\mu)} (G \textit{ is simple}) = o_n(1)$.
	\end{lemma}
	
	\noindent This lemma implies that as long as $E_{D \sim \mu} D^2 < \infty$, the results in Theorem \ref{thm:longSurvivalAlmostSIS} can be carried over to $\mathcal G(n, \mu)$. It is also established that $G\sim \mathcal{G}_{conf}(n,\mu)$ (and by extension, $G\sim \mathcal{G}(n,\mu)$) possesses a unique giant component of linear size w.h.p. if and only if $\E_{D\sim \mu} D(D-2) > 0$ \citep{mr95}. We assume this condition holds throughout our analysis, as it represents the most relevant regime; if $\E_{D\sim \mu} D(D-2) \leq 0$, the graph consists of fragmented, small components, making it impossible for the epidemic process to achieve an $e^{\Theta(n)}$ survival time for any $\lambda>0$.
	
	\subsection{The SIRS process}
	Let $G = (V, E)$ be an undirected graph with vertex set $V$ and edge set $E$. We equip the following Poisson point processes (PPPs) on the graph as follows.
	\begin{itemize}
		\item At each vertex $v \in V$, we put 2 PPPs $\mathfrak{R}_v \sim \PPP(1)$ for recovery and $\mathfrak{D}_v \sim \PPP(\alpha)$ for deimmunization. \footnote{$\mathfrak{R}$ is Fraktur R, and  $\mathfrak{D}$ is Fraktur D.}
		\item At each edge $(u, v) \in E$, we put 2 PPPs $\mathfrak{I}_{uv}, \mathfrak{I}_{vu} \sim \PPP(\la)$ for infection. \footnote{$\mathfrak{I}$ is Fraktur I.}
		\item All the PPPs $\{\mathfrak{R}_v\}_{v \in V}$, $\{\mathfrak{D}_v\}_{v \in V}$, $\{\mathfrak{I}_{uv}, \mathfrak{I}_{vu}\}_{(u, v) \in E}$ are mutually independent.
	\end{itemize}
	We sometimes refer to these processes as \textit{clocks}. We let all of the clocks above evolve simultaneously and independently, starting at time $0$. Almost surely, there is no time point at which the event time of two clocks happen at once. Also almost surely, there is a countably infinite number of event times of the collection of all clocks, which we index by the (random) increasing sequence $\{\gamma_i\}_{i \in \N}$\footnote{For infinite graphs, we do not have a countable number of event times in general; however, we do if the initial set of infections is finite. For contact processes with infinite initial configuration, there is a way to construct them using monotonicity and self-duality, but this is not the case for SIRS.}.
	
	An SIRS process $X = (X_t)_{t \ge 0}$ includes (i) an underlying graph $G = (V, E)$, (ii) an infection rate $\la$ and a deimmunization rate $\alpha$ (the recovery rate is normalized to $1$), and (iii) an initial partition $(S_0, I_0, R_0)$ of $V$ into susceptible, infected, and recovered vertices. At every time $t \ge 0$, $X_t$ is a partition of $V$ into  $(S_t, I_t, R_t)$. The configuration only changes at times in $\{\gamma_i\}_{i \in \N}$ as follows. The configuration transition at time $\gamma_i$ happens as follows.
	\begin{itemize}
		\item If for some $v \in V$, $\gamma_i \in \mathfrak{R}_v$ and $v \in I_{\gamma_{i-1}}$, then $S_{\gamma_i} = S_{\gamma_{i-1}}$, $I_{\gamma_i} = I_{\gamma_{i-1}} \setminus \{v\}$, and $R_{\gamma_i} = R_{\gamma_{i-1}} \cup \{v\}$. We say that $v$ \textit{recovers} at time $\gamma_i$.
		\item If for some $v \in V$, $\gamma_i \in \mathfrak{D}_v$ and $v \in R_{\gamma_{i-1}}$, then $S_{\gamma_i} = S_{\gamma_{i-1}} \cup \{v\}$, $I_{\gamma_i} = I_{\gamma_{i-1}}$, and $R_{\gamma_i} = R_{\gamma_{i-1}} \setminus \{v\}$. We say that $v$ \textit{becomes susceptible} at time $\gamma_i$.
		\item If for some $(u, v) \in E$, $\gamma_i \in \mathfrak{I}_{uv}$, $u \in I_{\gamma_{i-1}}$ and $v \in S_{\gamma_{i-1}}$, then $S_{\gamma_i} = S_{\gamma_{i-1}} \setminus \{v\}$, $I_{\gamma_i} = I_{\gamma_{i-1}} \cup \{v\}$, and $R_{\gamma_i} = R_{\gamma_{i-1}}$. We say that $u$ \textit{infects} $v$ at time $\gamma_i$.
	\end{itemize}
	If none of the above occurs, then the configuration at $\gamma_i$ of $X$ is the same as that at $\gamma_{i-1}$. For any vertex $v \in V$ and $t \ge 0$, we will write $S_t(v) = 1$ if $v \in S_t$, and $S_t(v) = 0$ otherwise. We define $I_t(v), R_t(v)$ similarly. In other words,
	$$ S_t(v) := 1_{\{v \in S_t\}}, \qquad I_t(v) := 1_{\{v \in I_t\}}, \qquad R_t(v) := 1_{\{v \in R_t\}}.$$
	
	\subsection{Graphical representation}\label{subsec:graphical}
	A standard construction of the SIRS process uses the \textit{graphical representation}. We briefly discuss this notion following \cite[Chapter 3, Section 6]{liggett:ips}. The idea is to record the infections, deimmunizations and recoveries of the SIRS process on a graph $G$ on the space-time domain $G \times \mathbb{R}_+$. Suppose we have all the clocks $\{\mathfrak{R}_v\}_{v \in V}$, $\{\mathfrak{D}_v\}_{v \in V}$, and $\{\mathfrak{I}_{uv}, \mathfrak{I}_{vu}\}_{(u, v) \in E}$ as in the previous subsection. Then the graphical representation is defined as follows:
	\begin{enumerate}
		\item Initially, we have the \textit{empty} domain $V \times \mathbb{R}_+$.
		
		\item For each $v\in V$, mark $\times$ at the point $(v,t)$, at each event time $t$ of $\mathfrak{R}_v(\cdot)$, and mark $\otimes$ at each event time of $\mathfrak{D}_v(\cdot)$.
		
		\item For each $\vec{uv}\in \overrightarrow{E}$, add an arrow from $(u,t)$ to $(v,t)$, at each event time $t$ of $\mathfrak{I}_{uv}(\cdot)$.
	\end{enumerate}

	\begin{figure}[H]
		\centering
		\begin{tikzpicture}[thick,scale=1.1, every node/.style={transform shape}]
			\foreach \x in {0,...,4}{
				\draw[black] (\x,0) -- (\x,3.8);
			}
			\draw[black] (0,0)--(4,0);
			\draw[dashed] (0,3.8)--(4,3.8);
			
			\draw (0,.8) node[cross=2.2pt,black]{};
			\draw (0,3.2) node[cross=2.2pt,black]{};
			\draw (1,2.4) node[cross=2.2pt,black]{};
			\draw (2,.4) node[cross=2.2pt,black]{};
			\draw (3,1.36) node[cross=2.2pt,black]{};
			\draw (4,.88) node[cross=2.2pt,black]{};
			\draw (4,2.8) node[cross=2.2pt,black]{};
			
			\draw (0,1.8) node[otimes]{}; 
			\draw (1,3.4) node[otimes]{};
			
			\draw (2,1.5) node[otimes]{};
			
			\draw[->] (0,2.96)--(.57,2.96);
			\draw[black] (0,2.96)--(1,2.96);
			\draw[->] (1,1.44)--(.43,1.44);
			\draw[black] (1,1.44)--(0,1.44);
			\draw[->] (1,2.08)--(1.57,2.08);
			\draw[black] (1,2.08)--(2,2.08);
			\draw[->] (2,1.2)--(1.43,1.2);
			\draw[black] (1,1.2)--(2,1.2);     
			\draw[->] (2,1.84)--(2.57,1.84);
			\draw[black] (2,1.84)--(3,1.84);
			\draw[->] (3,2.72)--(2.43,2.72);
			\draw[black] (2,2.72)--(3,2.72);       
			\draw[->] (3,.56)--(3.57,.56);
			\draw[black] (3,.56)--(4,.56);
			\draw[->] (3,3.28)--(3.57,3.28);
			\draw[black] (3,3.28)--(4,3.28);
			\draw[->] (4,2.16)--(3.43,2.16);
			\draw[black] (3,2.16)--(4,2.16);           
			
			\node at (4.7,0.1) {$t=0$};
			\node at (4.7,3.8) {$t=s$};
			\node at (0,-.3) {$1$};
			\node at (1,-.3) {$2$};
			\node at (2,-.3) {$3$};
			\node at (3,-.3) {$4$};
			\node at (4,-.3) {$5$};
			
			\draw [line width=.07cm, blue] (0,0) -- (0,.8);
			\draw [line width=.07cm, blue] (1,0) -- (1,2.4);   
			\draw [line width=.006cm, black] (1,1.44)--(0,1.44) -- (0,3.2);
			\draw [line width=.006cm, black] (0,2.96) -- (1,2.96)--(1,3.8);
			\draw [line width=.07cm, blue] (1,2.08) -- (2,2.08)--(2,3.8);  
			\draw [line width=.07cm, blue] (2,0) -- (2,.4);
			\draw [line width=.07cm, blue] (3,1.36) -- (3,0);
			\draw [line width=.07cm, blue] (4,0) -- (4,.88);
			\draw [line width=.07cm, blue] (3,.56) -- (4,.56);         
			
		\end{tikzpicture}
		\caption{A realization of the SIRS process on the path $V=\{1,\ldots,5\}$, $E = \{(i, i+1)\}_{1\le i \le 4}$, with all vertices initially infected. The blue lines describe the spread of infection.} \label{fig:SISfig}
	\end{figure}
	
	\noindent Figure \ref{fig:SISfig} tells us how to interpret the infections at time $t$ based on this graphical representation.
	\subsection{Probabilistic tools}
	We present several fundamental probabilistic inequalities and technical lemmas that we will use later. We start with an elementary lemma, whose proof is omitted.
	\begin{lemma}\label{lm:elementaryExponentialAlphaLambda1}
		Suppose $X, Y, Z$ are mutually independent and exponentially distributed with rates $\alpha, \la, 1$ respectively. Given $x > 0$, for $\alpha \neq 1, \la + 1$, we have 
		$$ \P(X + Y < x < X + Z) = \dfrac{\alpha e^{-x} - e^{-\alpha x}}{\alpha - 1} - \dfrac{\alpha e^{-(\la + 1) x } - (\la + 1)e^{-\alpha x}}{\alpha - (\la + 1)}.$$
	\end{lemma}
	
	We then present the following two lemmas, whose proof can be found in \cite{pemantle92}.
	\begin{lemma}\cite[Lemma 2.3]{pemantle92}\label{lm:PemLem23}
		Let $M$ be a positive integer-valued random variable. For any $x > 0$, let $M_x \sim \text{Bin}(M, x)$. Then for every $\delta > \max\{0, 1 - \E M\}$, there exists $\varepsilon > 0$ dependent on $\delta, x$ such that $$ \P(M_x \ge 1) \ge (1 - \delta)x\E M \wedge \varepsilon. $$
	\end{lemma} 
	
	\begin{lemma}\cite[Lemma 2.4]{pemantle92}\label{lm:PemLem24}
		Let $H$ be any non-decreasing function on $\R_{\ge 0}$ such that $H(x) \ge x$ on some neighborhood of $0$. Suppose $\phi$ is a function on $\R_{\ge 0}$ such that (i) $\inf_{0 \le t \le L} \phi(t) > 0$, and (ii) $\phi(t) \ge H\left( \inf_{0 \le s \le t - L} \phi(s) \right)$ for $t > L$. Then $\liminf \phi(t) > 0$.
	\end{lemma}
	
	We finally present the following lemma, which concerns the moments of maximum of independent exponential random variables.
	\begin{lemma}\cite[Section 1.1]{deHaan2006}\label{lm:maxOfExpRVs}
		Let $X_1, \dots, X_n$ be i.i.d. exponentially distributed with rate $\lambda$, and $X = \max(X_1, \dots, X_n)$. Then as $n \rightarrow \infty$, we have that $(X - \frac{1}{\la}\log n)$ converges in distribution to $G$ with cumulative distribution function $F_G(x) = \exp\left(- e^{-x}\right)$ (Gumbel distribution), whose all moments are finite. As a result, for all fixed positive integer $K$, there exists a constant $C_K > 0$ such that 
		$$ \E X^K \le C_K \la^{-K}(\log n)^K. $$
	\end{lemma}
	
	\section{Forced-infection process}\label{sec:forcedInfection}
	In this section, we introduce the forced-infection process, which is a monotone process that stochastically dominate the number of infected vertices in the original process from below. This process bypasses non-monotonicity issues of the SIRS process, and is used to prove persistence of SIRS.
	
		\begin{definition}[Forced-infection process at a subgraph]
		Consider the SIRS process on $G$, and suppose $H \subseteq G$ is a subgraph. Consider the following process, denoted $(C(H)_t)_{t \ge 0}$, defined on the graphical representation of $G$, which includes $H$.
		\begin{itemize}
			\item Keep all $(\mathfrak{R}_v)_{v \in V(H)}, (\frak{D}_v)_{v \in V(H)}$, i.e. the recovery and deimmunization clocks of $V(H)$.
			\item For all $u, v \in V(H)$, we keep the arrow from the infection clock $\mathfrak{I}_{uv}$ appearing at any time $t$ if at $v$, there exists an event from $(\mathfrak{R}_v \cup \frak{D}_v)$, and its last event right before $t$ comes from the deimmunization clock $\frak{D}_v$. Otherwise, we delete this arrow.
		\end{itemize} 
	\end{definition}
	
	We define $$\mathcal F(H) = \sigma\left((\mathfrak{R}_v)_{v \in V(H)}, (\frak{D}_v)_{v \in V(H)}, \{\mathfrak{I}_{uv}, \mathfrak{I}_{vu}\}_{(u, v) \in E(H)}  \right).$$
	Observe that $C(H)$ is measurable with respect to $\mathcal F(H)$. The following lemma states that within the considered subgraph $H$, infection in the SIRS process indeed stochastically dominate that of the forced-infection process from above at all times.
	\begin{lemma}
		For an SIRS process on a graph $G$ and $H \subseteq G$, let $(I|_H)_t \subseteq V(H)$ be the set of infected vertices of $H$ at time $t$. Moreover, let $I^H_t$ be the set of infected vertices of the forced-infection process of $H$ at time $t$. Then there is a coupling of the SIRS process on $G$ and the forced-infection process at $H$ such that for all time $t$, 
		$$ I^H_t \subseteq (I|_H)_t,$$
		i.e. if a vertex is infected in the forced-infection process, it is also infected in the original process.
	\end{lemma}
	\begin{proof}
		By coupling the forced-infection process of a graph $H$ with the SIRS process on its supergraph $G$ using the graphical representation, we prove this by induction on the event times. If for the next event, there is an infection successfully sent from $u \in I^H_t \subseteq (I|_H)_t$ to $v$ in the forced-infection process $C(H)$, it means that there is at least one event from $(\mathfrak{R}_v \cup \frak{D}_v)$ preceding the infection event, and the last event from this clock comes from $\frak{D}_v$. This means that in the SIRS process on $G$, right before time $t$, $v$ cannot be immune. Then by time $t$, either $v$ is already infected by another vertex in $G$, or it is newly infected by $u \in (I|_H)_t$. In either case, we have $v \in (I|_H)_t$, completing our proof.
	\end{proof}
	
	We will also include the following important lemma, whose proof is simply using the standard graphical representation, and will be omitted. 
	\begin{lemma}
		Consider the SIRS process on $G$, and suppose $H_1 \subseteq H_2 \subseteq G$ are subgraphs. Then there is a coupling of $C(H_1)$ and $C(H_2)$ such that for all time $t$, 
		$$ I_t^{H_1} \subseteq I_t^{H_2}.$$
		In other words, the forced-infection process is monotone.
	\end{lemma}
	
	From the definition of forced-infection process at a subgraph, we can also define a forced infection from one vertex to another, which only concerns the clocks at $v$ and the infection clock from $u$ to $v$. 
	
	\begin{definition}[Forced infection]
		We say that a vertex $u$ is \textit{forcibly infected} by a vertex $v \sim u$ at time $t$ if $v$ is already infected at time $t$, and $v$ infects $u$ in the coupled forced-infection process with $H$ being the subgraph consisting of $u$, $v$, and the edge between them.
	\end{definition}
	
	We provide a lower bound for the probability of a forced-infection, via the following lemma.
	\begin{lemma}\label{lm:forcedInfectionProb}
		Let $v \sim w$ be adjacent vertices on the graph $G$, and suppose $v$ is initially infected. Then the probability that $w$ is forcibly infected by $v$ is at least $\frac{\wal \wla}{1 + (1 - \wal)(1 - \wla)}$.
	\end{lemma}
	\begin{proof}
		For any time $t'$ when $v$ remains infected, let $\mathcal A(t')$ be the event that there is a forced infection from $v$ to $w$ looking only at the clocks from time $t'$ to the time that  $v$ recovers. We want to show
		$$p:= \P(\mathcal A(t') \mid v \text{ is infected at time } t')\ge \dfrac{\wal \wla}{1 + (1 - \wal)(1 - \wla)}.$$
		We note that $\mathcal A(t')$ is independent of the infection status of vertices in $G\setminus \{v\}$ at time $t'$, and so is $p$.
		
		To obtain a forced infection, we require an infection arrow from $v$ to $w$ to trigger immediately after a deimmunization clock at $w$, rather than after a recovery clock at that same vertex. For this reason, we say that $w$ is \textit{closed} initially and stays closed until a deimmunization clock at $w$ rings. It then stays \textit{open} until a recovery clock at $w$ rings, and so on. 
		
		Starting from time $t'$ with $w$ closed, two events can alter the status of $\mathcal{A}(t')$: either $v$ recovers first (making $\mathcal{A}(t')$ invalid thereafter), or the deimmunization clock at $w$ rings (which opens the vertex). This, together with the memoryless property of Poisson clocks, leads us to define
		\begin{align*}
			p_o &:= \P\left( \mathcal A(t') \mid w \text{ is open at time } t', v \text{ is infected at time } t'\right), \\
			p_c &:= \P\left( \mathcal A (t') \mid w \text{ is closed at time } t', v \text{ is infected at time } t' \right).
		\end{align*}
		Since among the above two clocks, the recovery clock at $v$ occurs at rate 1 while the deimmunization clock at $w$ occurs at rate $\alpha$, we get
		\begin{equation}\label{eq:pcInTermsOfpo}
			p_c= \wal p_o
		\end{equation}
		and 
		$$p\ge \min\{p_c, p_o\} = p_c.$$
		Similarly, if we are at a stage where $w$ is open, then the next non-deimmunization event would be one of the followings.
		\begin{itemize}
			\item The recovery of $v$, which happens at rate $1$. After that, $\mathcal A$ does not happen.
			\item An event from $\frak{R}_{w}$, which happens at rate $1$. The vertex $w$ becomes closed, and $\mathcal A$ from then happens with probability $p_c$.
			\item An event from $\frak{I}_{vw}$, which happens at rate $\lambda$. This is the forced infection and validate $\mathcal A$.
		\end{itemize}
		Combining the above, we have
		\begin{equation}\label{eq:poInTermsOfpc}
			p_o = \dfrac{1}{\la + 2} p_c + \dfrac{\la}{\la + 2}.
		\end{equation}
		Combining \eqref{eq:pcInTermsOfpo} and \eqref{eq:poInTermsOfpc} gives a lower bound on the probability of a single successful forced infection, 
		$$ p_c = \dfrac{\alpha\la}{(\alpha + 1)(\la + 1) + 1} = \dfrac{\wal \wla}{1 + (1 - \wal)(1 - \wla)},$$
		which is the desired lower bound.
	\end{proof}
	
	\section{Long survival for sparse random graphs}\label{sec:longSurvivalAlmostSIS}
	In this section, we prove Theorem \ref{thm:longSurvivalAlmostSIS}. To establish long-term persistence, we identify a high-degree core capable of sustaining the infection against the constraints of temporary immunity. Following the framework of \cite{bhamidi2021}, we rely on the existence of a robust $(\beta, R)$-embedded expander to prove that the infection is sustained for an exponentially long time. We then prove that for sufficiently large rates $\lambda$ and $\alpha$, an infection on a path is sustained long enough to propagate through the expander's structure.  
	
	\subsection{Embedded expanders}
	We recall the definition of a $(\beta, R)$-embedded expander in \cite{bhamidi2021}. 
	\begin{definition}
		Suppose $G_n = (V, E)$. For $\beta \in (0, 1)$ and $R > 0$, define a subset $W_0 \subseteq V$ an $(\beta, R)$-\textit{embedded expander} of $G_n$ if for all $A \subseteq W_0$ with $|A| \le \beta |W_0|$, we have $$ |N(A, R) \cap W_0| \ge 2|A|.$$ 
	\end{definition}
	
	We state the following structural lemma on the existence of giant embedded expanders in our random graph $G_n$.
	\begin{lemma}\cite[Lemma 5.2]{bhamidi2021}\label{lm:structuralLem1}
		Suppose $\mu$ satisfies the conditions in Theorem \ref{thm:longSurvivalAlmostSIS}, and $G_n \sim \mathcal{G}_{conf}(n, \mu)$. Then there exists positive constants $\beta, \gamma, R, j$ only dependent on $\mu$ such that the following holds w.h.p. There exists a subgraph $\overline{G}_n$ of $G$ whose maximal degree is at most $2j$, and an $(\beta, R)$-embedded expander $W_0$ of $\overline{G}_n$ such that $|W_0| \ge \gamma n$.
	\end{lemma}

	\subsection{Useful lemmas}
	We are now ready to state and prove the following lemmas. We first mention that our lemmas below typically includes $\overline{\mathcal F}_H$ in the conditional expectation, so that dynamics outside of the subgraph $H$ in consideration does not really affect what is going on inside $H$. 
	
	The first lemma states that when $\la$ and $\alpha$ are sufficiently large, w.h.p. infection on an edge is sustained for a long time, regardless of the dynamics outside of this edge.
	\begin{lemma}\label{lm:Pair:fastDeimLongSurvival}
		Consider the forced-infection process on some graph $G$. Let $(I_{t'})_{t' \ge 0}$ be the set of infected vertices, and $u \sim v$ be neighboring vertices in $G$, and let $H$ be the graph consisting of vertices $u, v$ and a single edge between them. For all $\varepsilon > 0$, there exists $c_{\varepsilon} > 0$ such that the following holds: on the event $\{\{u, v\}  \cap I_t \neq \emptyset \}$, for $\la\ge c_{\ep}$ and $\alpha \ge c_{\varepsilon}$,  we have, almost surely,
		$$ \pr\left(\mathcal A\ \bigg\vert\  \sigma(\mathcal F_t, \overline{\mathcal F}_{H}) \right) \ge 1 - \varepsilon,$$
		where $\mathcal A := \left\{ \{u, v\} \cap I_{t + t'} \neq \emptyset, \ \forall t' \in [0, K] \right\}$ is the event that $\{u, v\}$ is infected during the whole time interval and $K := \left( - \log \left( \frac{\wal \wla}{1 + (1 - \wal)(1 - \wla)} \right) \right)^{-1/3}$. In particular,
		$$\pr\left( \mathcal A\ \bigg\vert\ \{u, v\} \cap I_t \neq \emptyset \right) \ge 1 - \varepsilon.$$
	\end{lemma}
	
	\begin{proof}
		Without loss of generality, suppose $u \in I_t$. Let $\tau = \inf \{t' > 0: \{u, v\} \cap I_{t + t'} = \emptyset \}$. Then $\mathcal A := \left\{ \tau \ge K \right\}$. Moreover, let $(\tau_i)_{i \ge 0}$ be defined as the $i$-th time that a forced infection occurs between $u$ and $v$ while the one giving the forced infection remains infected. If this does not happen, we say that $\tau_i=\infty$. Finally, let $N$ be such that $\tau_N < \tau \le \tau_{N+1}$. In our calculations below, we omit the $\sigma$-algebras to simplify notations, and treat $K$ as an integer.
		
		Starting from any state with at least one of $u, v$ being infected, by Lemma \ref{lm:forcedInfectionProb}, the probability that there is a forced infection between them, say from $u$ to $v$ while $u$ remains infected is at least $\frac{\wal \wla}{1 + (1 - \wal)(1 - \wla)}$. We can thus think of $N$ as a geometric distribution with failing probability at most $\left(1 - \frac{\wal \wla}{1 + (1 - \wal)(1 - \wla)} \right)$. Therefore, for sufficiently large $\alpha$ and $\lambda$ (i.e. $K$ sufficiently large),
		\begin{equation}\label{eq:fastInfectionEq1}
			\pr(N < K^2) = 1 - \pr(N \ge K^2) \le 1 - \left( \dfrac{\wal \wla}{1 + (1 - \wal)(1 - \wla)} \right)^{K^2} = 1 - \exp\left(- 1/K \right) < \varepsilon/2.
		\end{equation}
		Now, if $\tau_i < \infty$, let $\tau_i' \in (\tau_i, \tau_{i+1})$ be the time that $\{u, v\}$ enters the state of one vertex being infected, and one vertex is immune/closed. Denote $\xi_i := \tau_i' - \tau_i$ if $\tau_i < \infty$, and is exponentially distributed with rate $2$ (independent of the SIRS process and $\overline{\mathcal F}_H$) if $\tau_i = \infty$. Intuitively, $\xi_i$ denotes the amount of time it takes for exactly one of the two infected vertices $u, v$ to recover. On the event $\{N \ge K^2\}$, we see that $\tau$ dominates the sum of $K^2$ i.i.d. $\sim \Exp(2)$ random variables, which is $ \text{Gam}(K^2, 2)$. For $K$ sufficiently large, by Paley-Zygmund inequality,
		\begin{equation}\label{eq:fastInfectionEq2}
			\pr(\tau < K, N \ge K^2) \le \pr\left(\text{Gam}(K^2, 2) < K \right)  <\varepsilon/2.
		\end{equation}
		Combining \eqref{eq:fastInfectionEq1} and \eqref{eq:fastInfectionEq2} completes our proof.
	\end{proof}
	
	If $\la$ and $\alpha$ are sufficiently large, we now extend the analysis to show that the disease can effectively traverse any path of bounded length.
	
	\begin{lemma}\label{lm:fastDeimWhpSurvival}
		Fix $R > 0$, and consider the forced-infection process on some graph $G$. Let $(I_{t'})_{t' \ge 0}$ be the set of infected vertices, $v = v_0, v_1, \dots, v_{R'} = u$ be a path in $G$ with $R' \le R$, and $H$ be the graph induced by this path. Then for all $\varepsilon \in (0, 1)$, there exist constants $c_0 > 1, C > 0$ only dependent on $R$ and $\varepsilon$ such that the following holds: on the event $\{v \in I_t\}$, for all $\alpha > c_0$ and $\la > c_0$, we have, almost surely,
		$$ \pr\left( u \in I_{t + C} \mid \sigma(\mathcal F_t, \overline{\mathcal F}_H) \right) \ge 1 - \varepsilon.  $$
		In particular, 
		$$ \pr\left( u \in I_{t + C} \mid v \in I_t \right) \ge 1 - \varepsilon.  $$
	\end{lemma}
	
	\begin{proof}
		First, if $v_i$ is infected at time $t$, then the probability that $v_{i+1}$ is infected at time $(t+1)$ can be bounded from below as follows. For sufficiently large $c_0$, the probability that $v_{i+1}$ is forcibly infected at some time within $[t, t+1]$ is at least 
		$$ \pr\left( \Exp(\alpha) + \Exp(\la) < 1 \wedge \Exp(2) \right) \ge e/3,$$
		where $\Exp(\alpha) + \Exp(\la)$ is the time it takes for $v_{i+1}$ to become open and be infected by $v_i$, and $\Exp(2)$ is the minimum between the duration of infection of $v_i$ (rate $1$), and the amount of time for $v_{i+1}$ to become closed (rate $1$) (we use the \textit{open/closed vertex} terminology as in the proof of Lemma \ref{lm:forcedInfectionProb}). Given this, the probability that $v_{i+1}$ stays infected until time $(t+1)$ is at least $e^{-1}$. Thus, doing this along the path of length $R' \le R$, the probability that $u$ is successfully infected at time $R'$ is at least $3^{-R'} \ge 3^{-R}$. For the rest of the calculations, we only restrict to the event $\{v \in I_t\}$.
		
		From Lemma \ref{lm:Pair:fastDeimLongSurvival}, for any large $C$, we can pick $c_0$ such that a.s.,
		$$ \pr\left( \{v, v_1\} \cap I_{t + t'} \neq \emptyset, \ \forall t' \in [0, C/3] \ \bigg\vert\  \sigma(\mathcal F_t, \overline{\mathcal F}_H) \right) \ge 1 - \varepsilon/4. $$
		By choosing $C$ large compared to $R$ and dividing the interval $[0, C/3]$ into intervals of length $R$ and consider $C/(3R)$ attempts of infecting $u$, we see that a.s.,
		\begin{align*}
			\pr\left(\exists t' \in [0, C/(3R)]: u \in I_{t + t'} \ \bigg\vert\  \sigma(\mathcal F_t, \overline{\mathcal F}_H) \right) &\ge \pr\left( \text{Bin}\left( C/(3R), 3^{-R} \right) \ge 1 \right) - \varepsilon/4 \ge 1 - \varepsilon/2.
		\end{align*}
		
		Consider the event where $u \in I_{t + t'}$ for some $t' \in [0, C/(3R)]$. Fix $u' := v_{R' - 1} \sim u$. By Lemma \ref{lm:Pair:fastDeimLongSurvival} (and increasing our choice of $c_0$ if needed), we see that with probability at least $(1 - \varepsilon/4)$, $\{u, u'\}$ contains an infection all throughout $[t + t',t + (t' + C)]$, which clearly contains $(t+C)$. We call this event $\mathcal A$ as in our proof of Lemma \ref{lm:Pair:fastDeimLongSurvival}.
		
		To prove that w.h.p. $u \in I_{t + C}$ on $\mathcal A$ is slightly trickier. The idea is to fix an interval of length $\ell$, where $\ell$ is appropriately small such that with high probability, the recovery clock at $u$ would not ring within this interval, but deimmunization and infection clocks would.  More precisely, consider the interval $[t + C - \ell, t + C]$ for some $\ell > 0$ to be chosen later. Throughout the calculations, we will increase $c_0$ if needed.
		\begin{itemize}
			\item[(i)] The probability that the recovery clock at $u$ would not ring within this interval is 
			$$ \pr(\Exp(1) > \ell) = e^{-\ell} \ge 1 - \varepsilon/16.$$
			
			\item[(ii)] The probability that there is at least one deimmunization clock in the first half $[t+C-\ell, t+C-\ell/2]$ is at least 
			$$ 1 - e^{- \ell \alpha/2} \ge 1 - e^{- \ell c_0/2} \ge 1 - \varepsilon/16. $$
			Similarly, the probability that an infection clock from $u'$ to $u$ rings at least once in the second half $[t+C-\ell/2, t+C]$ is at least $ 1 - \varepsilon/16$.

			\item[(iii)] If $u$ is infected at any time in $[t+C-\ell, t+C]$ then under (i), $u$ remains infected at time $t+C$. Otherwise, under $\mathcal A$, $u'$ must be infected throughout this interval. In that case, under (i) and (ii), any infection clock from $u'$ to $u$ in the second half interval is a forced infection which leads to $u$ being infected until time $t+C$.
		\end{itemize}
		
		From all the calculations above, on $\{v \in I_t\}$,
		$$\pr(u \in I_{t + C} \mid  \sigma(\mathcal F_t, \overline{\mathcal F}_H)) \ge 1 - \left(  \dfrac{\varepsilon}{2} + \dfrac{\varepsilon}{4} + 3 \cdot \dfrac{\varepsilon}{16}\right) > 1 - \varepsilon,$$
		as desired.
	\end{proof}
	
	\subsection{Proof of Theorem \ref{thm:longSurvivalAlmostSIS}}
	For the rest of this section, let $\beta, \gamma, j, R, \overline{G}_n, W_0$ be as in Lemma \ref{lm:structuralLem1}. Let $I_t^0 := I_t \cap W_0$ be the set of infected vertices in $W_0$ at time $t$. We will show that if $I_t^0$ is appropriately small, with overwhelmingly high probability, after some time $C$, the number of infected vertices in $W_0$ increases. 
	
	\begin{lemma}\label{lm:expanderIncreasesInfected1}
		Let $C$ and $c_0$ be constants corresponding to $\varepsilon = 1/4$ in Lemma \ref{lm:fastDeimWhpSurvival}. Then there exists a constant $C' = C'(j, R) > 0$ such that for $\alpha, \la > c_0$, and for all integers $a \in (0, \beta\gamma n]$,
		$$\pr\left( |I_{t+C}^0 | \le \dfrac{5a}{4} \ \bigg\vert\ \ |I_t^0| = a \right) \le 2\exp\left( - \dfrac{a}{C'} \right).$$ 
	\end{lemma}
	
	\begin{proof}
		Let $G_{n, t}$ be an induced subgraph of $\overline{G}_n$ on the set $\bigcup_{v \in I_t^0} N(v, R)$. Moreover, let $\mathcal X$ be the set of vertices $u \in I^0_{t+C}$ such that there exists $v \in I^0_t$ and a path of forced infection from $v$ at time $t$ to $u$ at time $(t + C)$, and all the vertices on the path lies in $N(v, R)$.
		
		Now, the number of vertices in $G_{n, t}$ is at most $a' \le a (2j)^{R+1}$, so we denote the vertices in $G_{n, t}$ by $v_1, \dots, v_{a'}$. Moreover, for each $i = 0, 1, \dots, a'$, let $\mathcal G_i$ be the $\sigma$-algebra generated by the randomness of the recovery clocks, deimmunization clocks, and infection clocks during time $(t, t+C]$ on the vertices $v_1, \dots, v_i$ and edges that connect them. Moreover, let $\overline{\mathcal F}_{G_{n,t}}$ be the $\sigma$-algebra generated by the randomness of all the clocks in $G$ during time $(t, t+C]$ that are not part of those generating $\mathcal G_{a'}$. Let 
		$$\mathcal X_i := \E\left( |\mathcal X| \ \bigg\vert \ \sigma(\mathcal G_i, \mathcal F_t, \overline{\mathcal F}_{G_{n,t}} ), |I_t^0| = a \right).$$
		
		Then $\mathcal X_{a'} = |\mathcal X|$. Since $(\mathcal X_i)_{0 \le i \le a'}$ is a (Doob) martingale, by Azuma's inequality, we have
		$$ \pr\left( |\mathcal X_{a'} - \mathcal X_0| \ge s \right)\le 2\exp\left( - \dfrac{s^2}{2a' K^2} \right) ,$$
		where $K := \max_{i} ||\mathcal X_{i+1} - \mathcal X_i||_{\infty} \le (2j)^{R + 1}.$
		
		Moreover, by Lemma \ref{lm:fastDeimWhpSurvival} and the expander properties of $W_0$ as in Lemma \ref{lm:structuralLem1}, we have (omitting the indicator on the event $\{|I_t^0| = a\}$)
		$$ \mathcal X_0 = \E\left( |\mathcal X| \ \bigg\vert\ \sigma(\mathcal F_t, \overline{\mathcal F}_{a'}), |I_t^0| = a \right) \ge  \dfrac{3}{4}| N(I^0_t, R) \cap W_0| \ge \dfrac{3}{4} \cdot 2 |I^0_t| = \dfrac{3a}{2}.$$
		
		Substituting $s = a/4$ and note that $\mathcal X_{a'} \le |I_{t + C}^0|$, our proof is completed.
	\end{proof}
	
	The lower bound for survival time in Theorem \ref{thm:longSurvivalAlmostSIS} can be established from Lemma \ref{lm:expanderIncreasesInfected1} as follows. Define $m :=  \exp\left( \frac{\beta\gamma}{2C'} n\right)$, and suppose that $m$ is an integer to simplify the argument. For the lower bound, let $|I^0_0| = |W_0| > \beta \gamma n$. Let $t_1 > 0$ be the first time that $|I_{t_1}^0| = \beta\gamma n$. By Lemma \ref{lm:expanderIncreasesInfected1}, for some constant $C$, at time $t'_2 = t_1 + C$, we have $|I^0_{t_2'}| \ge \frac{5}{4}\beta\gamma n$ with probability at least $1 - 2 \exp\left( - \frac{\beta\gamma}{C'} n \right) = 1 - 2/m^2$. Then we let $t_2 > t_2'$ be the first time after $t_2'$ that $|I_{t_2}^0| = \beta \gamma n$. Repeating this process $m$ times, then with probability at least $1 - 2/m = 1- 2\exp\left( -\Theta(n) \right)$, the process survives with time at least $mC = C\exp\left( \frac{\beta\gamma}{2C'} n\right)$, completing the lower bound. For the upper bound, a standard coupling between the contact process and SIRS process on the same graphical model immediately shows that the infected vertices in the contact process stochastically dominates that of the SIRS process from above. The well-known $\exp(\Theta(n))$ upper bound for survival of the contact process immediately implies that for the SIRS process.

	\section{Long survival for heavy-tailed sparse graphs}\label{sec:noSubcritPhaseInHeavyTail}
	In this section, we prove Theorem \ref{thm:noSubcritPhaseInHeavyTail}. Here, to show long survival for all $\alpha$ and $\lambda$, we need a stronger structure than embedded expanders as in the previous section. We show that in heavy-tailed networks, there exists a set of high-degree vertices $W$ with degrees in the range $[M, f(M)]$ in which  any small subset of infected vertices has an immediate neighborhood significantly larger than itself.

	To simplify notations for the remainder of this section, we define, for each $M > 0$, 
	$$ \frak{u}_M := \mu([M, f(M)]). $$
	
	\subsection{Structural lemmas for heavy-tailed sparse graph}
	It is well-known that with high probability heavy-tailed graphs $G_n$ contains a giant $K$-core for any $K$ (for example, see \cite{fernholz03}). A $K$-core of a graph is an induced subgraph such that all vertices have (induced) degree at least $K$. The main lemma of this section asserts that we can obtain something even finer: in fact,  we can find a giant core with well-behaved internal degrees and good expansion properties.
	
	\begin{definition}
		Suppose $G = (V, E)$. For $\beta \in (0, 1)$ and $C > 0$, define a subset $W_0 \subseteq V$ a $(\beta, C)$-\textit{vertex expander} of $G_n$ if for all $A \subseteq W_0$ with $|A| \le \beta |W_0|$, we have 
		$$ |N(A, 1) \cap W_0| \ge (1 + C)|A|. $$
	\end{definition}
	The following lemma asserts that we can find a good vertex expander within $G_n$.
	
	\begin{lemma}\label{lm:structuralLemmaExpander}
		Suppose $G_n \sim \mathcal G_{conf}(n, \mu)$, where $\mu$ satisfies the conditions in Theorem \ref{thm:noSubcritPhaseInHeavyTail}. Then for every $K > 0$, there exist $\beta, \gamma \in (0, 1)$ and $K' > K$ (dependent on $K, \mu$) such that w.h.p., there exists a subgraph $W_0 \subseteq G_n$ such that the following hold:
		\begin{itemize}
			\item $W_0$ has at least $\gamma n$ vertices.
			\item For all $v \in W_0$, $\deg_{W_0} v \in [K, K']$, and
			\item $W_0$ is a $(\beta, K)$-vertex expander.
		\end{itemize} 
	\end{lemma}
	High-degree cores in a graph are typically found via peeling algorithms. To obtain finer information about expansion properties of this core, we first find a subset of sufficiently high-degree vertices $W$ from $G_n$. For vertices in $W$, color exactly $M$ half-edges \textit{blue} (where $M$ is chosen according to $K$ later). When exploring $1$-neighborhoods of vertices in $W$ simultaneously, we call an edge \textit{blue} if it is created by matching two \textit{blue} half-edges. The idea is to prove that w.h.p. we have enough blue edges so that roughly speaking, vertices in $W$ have high internal degree, and we can extract a high-degree core $W_0$ from $W$ accordingly. At the core of this scheme is the so-called cut-off line algorithm, which gives us a construction of the configuration model. 
	
	\begin{definition}[Cut-off line algorithm]\label{def:cutofflinealgo}
		Given a graph $G_n$ in which each vertex $v$ has degree $d_{G_n}(v)$, a perfect matching of the half-edges of $G_n$ is obtained through the following algorithm:
		\begin{itemize}
			\item Each half-edge of a vertex $v$ is independently assigned a height uniformly chosen in $[0, 1]$, and is placed on the line of vertex $v$.
			\item Set the initial cut-off line at height $1$.
			\item At each step, pick an unmatched half-edge independent of the heights of all unmatched half-edges and match it to the highest unmatched half-edge. Move the cut-off line to the height of the latter half-edge.
		\end{itemize}
	\end{definition}
	
	The rest of this subsection is devoted to proving Lemma \ref{lm:structuralLemmaExpander}. We start with the following useful lemma. 
	\begin{lemma}\label{lm:preprocessingForPureExpander}
		Suppose $G_n \sim \mathcal G_{conf}(n, \mu)$, where $\mu$ satisfies the conditions in Lemma \ref{lm:structuralLemmaExpander}. Define $$W := \left\{ v \in G_n: \deg_{G_n} v \in [M, f(M)] \right\}, $$
		for any fixed $M$. Then the following holds w.h.p.
		\begin{itemize}
			\item[(i)] The total degree in $G_n$, denoted $d_{G_n}$, is at most $2nd$.
			\item[(ii)] The number of vertices in $W$ is at least $ n\frak{u}_M/2 $.
		\end{itemize}
	\end{lemma}
	With this lemma, we will later choose $M$ sufficiently large depending on $K$.
	
	\begin{proof}
		The proof of the first item is standard: since $\{\deg_{G_n}(v)\}_{v \in V(G_n)}$ are i.i.d. $\sim \mu$, by the weak Law of Large Numbers, $\frac{d_{G_n}}{n}\overset{p}{\to}d$, so
		$$ \P\left( d_{G_n} > 2nd \right) \to 0.$$
		
		The second item follows from the fact that $|W| \sim \text{Bin}(n, \frak{u}_M)$, so we can use Chernoff inequality.
	\end{proof}

	Consider the following coloring scheme: for each vertex in $W$, we pick exactly $M$ half-edges among all those of $v$ uniformly at random, then color them \textit{blue}. We perform uniform random matching among all half-edges in $G$, and slightly abusing notation, we use $W$ once again to denote the induced subgraph on vertices in $W$. Denote $\frak{B}$ the number of \textit{blue} edges (edges formed by two \textit{blue} half-edges), and for each $v \in W$, let $\deg_b(v)$ be the number of \textit{blue} edges adjacent to $v$.
	
	Conditioned on $\frak{B}$, the distribution of $\{\deg_b(v)\}_{v \in W}$ is given by 
	\begin{center}
		$\{B_v\}_{v \in W},  B_v \sim \text{i.i.d. } \text{Bin}\left(M, \dfrac{2\frak{B}}{M|W|}\right)$ conditioned on $\sum_{v \in W} B_v = 2\frak{B}$.
	\end{center}
	Define $\theta:= \dfrac{2\frak{B}}{M|W|}$. We first show that w.h.p. this success probability is sufficiently large.
	\begin{lemma}\label{lm:probOfBlueEdgeIsHigh}
		With high probability, $\theta \ge \dfrac{\frak{u}_M M}{24d}$.
	\end{lemma}
	
	\begin{proof}[Proof of Lemma \ref{lm:probOfBlueEdgeIsHigh}]
		We have established that $|W| \ge n\frak{u}_M/2$. Thus, w.h.p. the number of blue half-edges is
		at least $M |W| \ge n\frak{u}_M M/2$, while the total degree of $G_n$ is at most $2nd$ by Lemma \ref{lm:preprocessingForPureExpander}. Let 
		$$ \theta_0 := \dfrac{\#\{\text{blue half-edges in } W\}}{\#\{\text{half-edges in } G_n\}} =: \dfrac{d_W}{d_{G_n}} \ge \dfrac{u_M M}{4d}. $$
		
		We now show that w.h.p., we have $\frak{B} \ge \theta_0 d_W/12$, which completes our proof since $d_W = M |W|$. We will split the set of blue half-edges of $W$ into two parts $A, B$ of equal size (independent of their height). First, we match the half-edges in $A$. Then at least $d_W/4 = \theta_0 d_{G_n}/4$ highest half-edges of $G_n$ are matched during this step. The cut-off line is then of height at most $(1 - \theta_0/5)$: this follows by applying Chernoff bounds for the probability that $\text{Bin}\left( d_{G_n}, \theta_0/5 \right) \ge \theta_0 d_{G_n}/4$. 
		
		Similarly, the number of half-edges in $B$ that lies above the cut-off $(1 - \theta_0/5)$ is at least $\theta_0 d_{W}/12$ w.h.p. (by applying Chernoff bounds for $\text{Bin}\left( d_{W}/2, \theta_0/5 \right)$). Thus, w.h.p. $\frak{B} \ge \theta_0 d_{W}/12$, completing our proof.
	\end{proof}
	
	From the lemma above, we will find a high-degree core of $W$.
	
	\begin{lemma}\label{lm:StartToPickHighDegCore}
		For sufficiently large $M$, let $\theta$ be as in Lemma \ref{lm:probOfBlueEdgeIsHigh}. Suppose $\theta M \ge 200(1 + K)$. Then w.h.p. we can find $W_0 \subseteq W$ such that the following holds.
		\begin{itemize}
			\item[(i)] The number of vertices in $W_0$ is at least $|W|/2$.
			\item[(ii)] $W_0$ is a $\frac{\theta M}{20}$-core.
			\item[(iii)] Each vertex in $W_0$ has degree at most $f(M)$.
		\end{itemize}
	\end{lemma}
	\begin{proof}
		We see that (iii) simply follows from the definition of $W$. Let $s := \theta M / 20$. 
		
		For the rest of this proof, we only look at the blue half-edges in $W$ that are matched to another blue half-edge in $W$. The idea is to prune out low-(blue)-degree vertices to get a high-degree core, using the cut-off line algorithm as follows. Each blue half-edge is reassigned a height uniformly chosen in $[0, 1]$. If there is a vertex with less than $s$ unmatched half-edge (i.e. less than $s$ half-edge below the cut-off line), then match its half-edges to the highest unmatched half-edges and move the cut-off line accordingly, then prune out this vertex. We repeat this until there are no such vertices left. 
		
		Let $W_0$ be the set of remaining vertices; it is clear that vertices in $W_0$ have internal degree at least $s$, so item (ii) is satisfied. It remains to show item (i), i.e. w.h.p. $|W_0| \ge |W|/2$. First, we would like to get rid of the conditioning on the event $\{\sum_{v \in W} B_v = 2\frak{B}\}$, so that each $B_v$ can be treated as i.i.d. Let us forget about the definition of $\frak{B}$, and first fix $\theta$ that satisfies Lemma \ref{lm:probOfBlueEdgeIsHigh}. If $\{B_v\}_{v \in W}$ are i.i.d. $\text{Bin}(M, \theta)$, then $\sum_{v \in W} B_v \sim \text{Bin}(M |W|, \theta)$. Since $M |W| \theta ( 1 - \theta) = \Theta(n) \gg 1$, using Sterling's approximation, we have that $$\P\left( \sum_{v \in W} B_v = M |W| \theta \right) \sim \left(2\pi M |W| \theta( 1- \theta) \right)^{-1/2} = \Theta(n^{-1/2}).$$
		
		We will see that after removing this conditioning, all the subsequent tail events we compute will be exponentially small in $n$, so we can remove this conditioning and view $B_v \sim \text{Bin}(M, \theta)$ as i.i.d.
		
		We show that after the removal, w.h.p. the cut-off line is at least $2/3$. Assuming this, for each $v \in W$, the probability that it is pruned out is at most the probability that it has less than $s$ unmatched half-edges below the line $2/3$, which is $\{\text{Bin}(M, 2\theta/3) \le s\}$. By Chernoff inequality and the assumption that $\theta M \ge 200(1 + K) > 200$, this event happens with probability at most $1/200$. Thus, each vertex has a $199/200$ probability of being retained, so once again by Chernoff's inequality, w.h.p. $|W_0| \ge |W|/2$.
		
		To show that w.h.p. the cut-off line is at least $2/3$, let $a$ be the number of removed vertices ($0 \le a \le |W|$). Each time we prune out a vertex, we match at most $2s$ half-edges, so there are at most $2as \le 2s|W|$ half-edges above the cut-off line. On the other hand, for each vertex $v$, the number of half-edges above the cut-off line of $2/3$ follows a $\text{Bin}(M, \theta/3)$ distribution, so by Chernoff's inequality, w.h.p. the number of half-edges above the line $2/3$ is at least $ (M \theta/4)|W| = 5s|W| > 2s|W| $, completing our proof.
	\end{proof}
	
	We will now complete the proof by showing that $W_0$ is our desired expander. Note that we can then pick $M$ sufficiently large: by Lemma \ref{lm:probOfBlueEdgeIsHigh},
	\begin{equation}\label{eq:choiceOfLargeMwrtKandC}
		\dfrac{\theta M}{20} \ge \dfrac{M^2 \frak{u}_M}{20 \cdot 24 d} > 10(1 + K),
	\end{equation}
	for arbitrary large $M$. We can simply pick $K' = f(M)$.
	
	\begin{lemma}
		For any fixed $K > 0$, let $M$ and $\theta$ satisfies \eqref{eq:choiceOfLargeMwrtKandC} and Lemma \ref{lm:probOfBlueEdgeIsHigh}, respectively. Then there exists a constant $\beta > 0$ such that w.h.p. $W_0$ is a $(\beta, K)$-vertex expander.
	\end{lemma}
	
	\begin{proof}
		Denote $N := |W_0| \ge n \frak{u}_M/4$, and $s := \frac{\theta M}{20} \ge 10(1 + K)$. It suffices to show that for any $A \subseteq W_0$ of size at most $\beta N$, the size of $|N(A, 1)\cap W_0|$ is at least $(1 + K)$ times that of $A$, i.e. for all $m \le \beta N$ and $A, B \subseteq W_0$ with $|A| = m, |B| = (1 + K)m$, the neighbors of $A$ are not fully contained in $B$. Fix two sets $A, B \subseteq W_0$ with $|A| = m, |B| = (1 + K)m$. Then by Lemma \ref{lm:StartToPickHighDegCore},
		\begin{itemize}
			\item The number of half-edges in $A$ is at least $ms$,
			\item The number of half-edges in $B$ is at most $b \le (1 + K)K' m$,
			\item The total number of half-edges in $W_0$ is $c \ge Ns$.
		\end{itemize}
		Then the probability of all neighbors of $A$ belonging to $B$ is at most
		$$ \prod_{i = 1}^{ms}\dfrac{b - (i - 1)}{c - (2i - 1)} \le \left( \dfrac{(1 + K) K' m}{Ns} \right)^{ms}.$$
		Taking the union bound over $m$ and all choices of $A, B$, then the probability that $W_0$ is not a $(\beta,K)$-vertex expander is at most
		\begin{align}
			&\sum_{m = 1}^{\beta N} \binom{N}{m}\binom{N}{(1 + K)m} \left( \dfrac{(1 + K) K' m}{Ns} \right)^{ms} \nonumber \\
			\le &  \sum_{m = 1}^{\log n} C_0 \left( C_1\left(\dfrac{\log n}{N}\right)^{s - (2 + K)} \right)^{m} + \sum_{m = \log n + 1}^{\beta N} (C_2 \beta^{s - (2 + K)})^m, \label{eq:BoundBadProbOfExpander} 
		\end{align}
		for constants $C_0, C_1, C_2$ only dependent on $K, K', s$ since $s - (2+K)>0$. Choosing $\beta$ sufficiently small makes \eqref{eq:BoundBadProbOfExpander} of order $o(1)$, completing our proof.
	\end{proof}

	\subsection{Dynamics within the Expander} 
	To establish the long-term persistence of the forced-infection process on heavy-tailed networks, we first quantify the local transition probabilities that allow the infection to propagate through the high-degree hubs of the vertex expander.
	\begin{lemma}\label{lm:fastDeimWhpSurvival2}
		There exist $p \in (0, 1)$ dependent on $\alpha, \lambda$ such that the following holds. Let $W_0$ be the expander found in Lemma \ref{lm:structuralLemmaExpander}. For all $u \sim_{W_0} v$, let $H$ be the subgraph of $W_0$ created by $u$, $v$, and a single edge between them. Then on the event $\{v \in I_t\}$, almost surely,
		
		$$ \pr\left(u \in I_{t + 1} \mid \sigma(\mathcal F_t, \overline{\mathcal F}_H)\right) \ge p.$$
	\end{lemma}
	The proof of this is standard: we calculate the probability that $v$ is infected throughout $[t, t+1]$, bound the probability that there is a forced infection being sent from $v$ to $u$ within $[t, t+1]$, then bound the probability that $u$ stays infected until time $(t + 1)$. 
	
	From here, it is rather straightforward to prove Theorem \ref{thm:noSubcritPhaseInHeavyTail}. Let $p$ be chosen according to Lemma \ref{lm:fastDeimWhpSurvival2}, and $K := \frac{3}{p} - 1 > 0$. From there, let $\beta, \gamma, K', W_0$ be chosen according to Lemma \ref{lm:structuralLemmaExpander}. Finally, let $I^0_t := I_t \cap W_0$. Then Theorem \ref{thm:noSubcritPhaseInHeavyTail} follows directly from the following analogue of Lemma \ref{lm:expanderIncreasesInfected1}.
	
	\begin{lemma}
		There exists a constant $C' = C'(\lambda, \alpha, \mu) > 0$ such that for all integers $a \in (0, \beta\gamma n]$,
		$$\pr\left( |I^0_{t + 1}| \le 2a \ \bigg\vert \ |I_t^0| = a \right) \le 2\exp\left( - \dfrac{a^2}{C'n } \right).$$
	\end{lemma}
	The proof is essentially the same as that of Lemma \ref{lm:expanderIncreasesInfected1}: define $\mathcal X$ to be the set of vertices $u \in W_0$ such that there exists $v \in I_t^0$, $v \sim u$, and a forced infection from $v$ at time $t$ to $u$ at time $(t + 1)$. We construct $(\mathcal X_i)_{0 \le i \le a'}$ as before and use Azuma's inequality to obtain
	\begin{align*}
		\P\left(|\mathcal X_{a'} - \mathcal X_0| \ge s \right) &\le 2\exp\left( - \dfrac{s^2}{2|W_0| K'^4}\right) \le 2 \exp\left( - \dfrac{s^2}{2 n K'^4 }\right), \\
		|\mathcal X_0| &\ge p \left| N(I_t^0, 1) \cap W_0 \right| \ge p (1 + K) |I_t^0| = 3a.
	\end{align*}
	Substituting $s = a$ into the above completes our proof.

	\section{Strong survival for large trees}\label{sec:mainThmOfStrongSurvival}
	In this section, we prove Theorem \ref{thm:mainThmOfStrongSurvival}, using a strategy in \cite{pemantle92} and \cite{huangdurrett20}. We indeed prove a stronger statement that with positive probability, the root is infected infinitely many times in the forced-infection process. For the rest of the proof, this is the process that we consider. We establish that the root of a Galton-Watson tree is re-infected infinitely often by showing that the survival probability function $\phi(t)$ satisfies $\liminf_{t\to\infty} \phi(t) > 0$. With modifications made along the way, to lower bound the probability that $\rho$ is infected at any sufficiently large time $t$, the general outline of our strategy is as follows:
	\begin{enumerate}
		\item We condition on the root having $n$ children, where $n$ is appropriately large. The infection survives in the neighborhood $N(\rho)$ of the root $\rho$, and $\rho$ is re-infected sufficiently often within a large time interval $[0, L_1]$ (Lemma \ref{lm:tailBoundForTauS}).
		\item During that time $[0, L_1]$, there are enough infected vertices with $n$ children at some $r$-th generation. 
		\item At some time $(t - L_2)$ when at least one of the high-degree vertex at level $r$ is infected, the infection there is then sustained for a sufficiently long period of time so that it is successfully pushed back to $\rho$ within time $(t - L_2, t]$.
		\item Finally, we show that the root is infected at time $t$ with some positive probability (Lemma \ref{lm:constantProbHelperRoot}).
	\end{enumerate}
	
	For technical reasons, we define $\phi(t)$ as the infimum probability that the root is infected within a look-back window of $2T$, where $T = \Theta_{\alpha, \varepsilon_0}\left( (\wla^2 n)^{\alpha\varepsilon_0} ( - \log \mu([n, \infty)) + \log n)  \right)$, with $\varepsilon_0$ chosen appropriately later. Instead of finding deterministic $L_1, L_2$, we would appropriately modify them to be look-back windows with constant lengths. Our goal is to apply Lemma \ref{lm:PemLem24}, with appropriately chosen $n$. We note that the $\inf_{x \le t - L}\phi(x)$ term inside the function $H$ comes from item (3) in the above. Our choice of the function $H(x)$ will have the form $$H(x) := Ax \wedge \varepsilon',$$
	where $A > 1$. We will show that it is possible to choose such $A$ in the setting of Theorem \ref{thm:mainThmOfStrongSurvival}.
	
	Let us further elaborate on our strategy. Let $\Omega_{\rho} \subseteq \Omega$ be the set of all configurations of $\T(\rho, \mu)$ with finite infections such that $\rho$ is infected. For a parameter $T$ to be specified later, we define the "worst-case" probability of the root being infected in a $2T$-time window to be
	\begin{align*}
		\phi(t) := \begin{cases}
			\inf_{\T}\inf_{\mathcal C_0 \in \Omega_{\rho}} \P\left(\exists s \in [t - 2T, t]: \rho \in I^{\mathcal T}_s \ \bigg\vert \ |N_{\mathcal T}(\rho)| \ge n\right) &  \text{ if } t > 2T, \\
			1 &  \text{ if }  t \le 2T,
		\end{cases}
	\end{align*}
	where $I^{\mathcal T}_t$ denotes the set of infected vertices in the forced-infection process on $\mathcal T$, and $N_{\mathcal T}(\rho)$ denotes the number of children of $\rho$ in the appended Galton-Watson tree $\mathcal T$. The reason why we use $(I^{\mathcal T}_t)$ instead of the true infection $(I_t)$ is to make $\phi$ a lower bound as universal as possible. 

    \subsection{A priori estimates of forced-infections on paths}
	
	We now start with the following lemma, which gives us a lower bound for the probability of infecting a vertex of distance $r$ away within our graph.
	\begin{lemma}\label{lm:probInfectFarWithinShort}
		Let $v_0, v_1, \dots, v_r$ be a path in any graph $G$, we run the forced-infection process on $G$, and suppose $v_0$ is infected at a given time $t$. Then the probability that there exists a chain of infections from $v_0$ to $v_r$ by time $t+3r$
		is at least $$\left( \dfrac{\wal \wla}{1 + (1 - \wal)(1 - \wla)} \right)^r (1 - \exp(-\gamma r))  \ge \dfrac{1}{2}\left( \dfrac{\wal \wla}{1 + (1 - \wal)(1 - \wla)} \right)^r ,$$
		for some numerical constant $\gamma > 0$ and $r > r_0,$ where $r_0$ is also a numerical constant. 
	\end{lemma}
	
	\begin{proof}
		The term $\left( \frac{\wal \wla}{1 + (1 - \wal)(1 - \wla)}\right)^r$ directly comes from Lemma \ref{lm:forcedInfectionProb}, so it remains to prove that conditioned on having a forced-infection chain from $v_0$ to $v_r$, the total wait time for this is at most $3r$ with probability at least $(1 - \exp(-\gamma r))$.
		To this end, we will show that for a single forced infection, conditioned on the event that it is successful, the distribution of the amount of time follows a hypoexponential distribution $\text{Hypo}(\la_{-}, \la_+)$, where $$ \la_{\pm} = \dfrac{\alpha + \la + 3 \pm \sqrt{(\la + 1 - \alpha)^2 + 4\alpha}}{2}.$$
		Assuming this, then we can apply a Chernoff bound on the sum of $r$ independent $\text{Hypo}(\la_{-}, \la_+)$-distributed random variables, noting that a $\text{Hypo}(\la_{-}, \la_+)$ distribution has mean
		$$\dfrac{1}{\la_-} + \dfrac{1}{\la_+} = \dfrac{\la_+ + \la_-}{\la_+ \la_-} = \dfrac{\alpha + \la + 3}{(\la + 1)(\alpha + 1) + 1}\le \dfrac{3}{2} < 3,$$
		completing our proof. Now, we fix a time $t'$ when $v_{i-1}$ is infected and let $\tau$ be the amount of time until the first forced infection from $v_{i-1}$ to $v_i$ before $v_{i-1}$ recovers. Let $\mathcal B_o, \mathcal B_c$ denote the events that $v_i$ is open/closed respectively at time $t'$. Define the following Laplace transforms
		$$ \phi_o(s) = \E_{\mathcal B_o}\left( e^{-s\tau} 1_{\{\tau < \infty\}} \right), \qquad \phi_c(s) = \E_{\mathcal B_c}\left( e^{-s\tau} 1_{\{\tau < \infty\}} \right). $$
		We now aim to express $\phi_c$ in terms of $\phi_o$. To do this, let $D_1 \sim \Exp(\alpha)$ and $R_1 \sim \Exp(1)$ be the time until the an event time of $\frak{D}_{v_i}$, and the time until the recovery of $v_{i-1}$ respectively. Thus, the first relevant waiting time is $S_o = \min(D_1, R_1) \sim \Exp(\alpha + 1)$, and the probability that $D_1 < R_1$ is $\wal$. If $R_1 \le D_1$, then $\tau = \infty$, and its contribution to the integrand inside $\phi_c(s)$ is $0$. If $D_1 < R_1$, then after time $S_o$, the process "restarts" at state $\mathcal B_o$. This means 
		$$\tau = S_o + \tau',$$
		where $\tau'$ is a copy of $\tau$ but starting from the state $\mathcal B_o$ instead of $\mathcal B_c$, independent of $S_o$. This means
		\begin{equation}\label{eq:phicInTermsOfphio}
			\phi_c(s) = \E_{\mathcal B_c}\left(e^{- s(S_o + \tau')}1_{\{ D_1 < R_1\}}1_{\{\tau' < \infty\}}  \right) = \E \left( e^{-sS_0} 1_{\{D_1 < R_1\}}\right) \phi_o(s) = \dfrac{\alpha}{s+ \alpha + 1}\phi_o(s),
		\end{equation}
		the last equality is obtained by a routine, direct calculation of $\E \left( e^{-sS_0} 1_{\{D_1 < R_1\}}\right)$. Similarly,
		\begin{equation}\label{eq:phioInTermsOfphic}
			\phi_o(s) = \dfrac{\la}{s + \la + 2} + \dfrac{1}{s + \la + 2}\phi_c(s).
		\end{equation}
		Together, \eqref{eq:phicInTermsOfphio} and \eqref{eq:phioInTermsOfphic} imply that
		$$\phi_c(s) = \dfrac{\alpha \la}{(s + \alpha + 1) (s + \la + 2) - \alpha} = p_c \dfrac{(\alpha + 1) (\la + 1) + 1}{(s + \alpha + 1) (s + \la + 2) - \alpha} = p_c\dfrac{\la_- \la_+}{(s + \la_-)(s + \la_+)}.$$
		In other words, we have the conditional Laplace transform
		$$\E_{\mathcal B_c}\left( e^{-s\tau} \ \big\vert\ \tau < \infty \right) = \dfrac{\E_{\mathcal B_c}\left( e^{-s\tau} 1_{\{\tau < \infty\}} \right)}{\P\left( \tau < \infty \ \big\vert\ \mathcal B_c \right)} = \dfrac{\phi_c(s)}{p_c} = \dfrac{\la_- \la_+}{(s + \la_-)(s + \la_+)},$$
		which is exactly the Laplace transform of $\text{Hypo}(\la_{-}, \la_+)$. This completes our proof.
	\end{proof}
	
	\subsection{A priori estimates of forced-infections on stars}
	Using Lemma \ref{lm:elementaryExponentialAlphaLambda1}, we observe that the concept of $\varepsilon$-good rounds in \cite{lam2024optimal} still applies to the setting of Galton-Watson trees. We define rounds in terms of the forced-infection process in the root $\rho$: let $0 = \tau_1 < \tau_2 < \dots$ be the (random) times at which the root $\rho$ is reinfected before the infection disappears from the neighborhood $N(\rho)$. For the $i$-th round $[\tau_i, \tau_{i+1})$, denote by $\tau_i^R, \tau_i^S$ the time the root $\rho$ becomes recovered and susceptible, respectively. Then we say that a round is $\varepsilon$-good if $\tau_i^R - \tau_i \ge \varepsilon$, i.e. the root $\rho$ stays infected for some time at least $\varepsilon$ before its recovery.
	
	With this notion, for each neighbor $v$ of $\rho$, let $p_{v, R \mapsto I}(x)$ be the infimum probability that, conditioned on the event that $v$ is immune at time $t$, and the root $\rho$ is infected throughout $[t, t+x]$, the vertex $v$ is infected at time $t + x$. The infimum is taken over all possible configurations $\Omega := \{S, I, R\}^{V(\T(\rho, \mu))}$, so conditioned on $\mathcal T$, $p_{v, R \mapsto I}(x)$ is deterministic, thus $p_{v, R \mapsto I}(x) \in \mathcal F_0$. Formally, $$p_{v, R \mapsto I}(x) := \inf_{t \ge 0}\inf_{X_t \in \Omega} \P\left(v \in I^{\mathcal T}_{t + x} \mid v \in R_t, \rho \in I^{\mathcal T}_{s}\ \forall s\in [t, t+x] \right). $$
	
	We then define 
	$$ p_{R \mapsto I}(x) :=  \inf_{v \sim \rho} p_{v, R \mapsto I}(x). $$
	
	We define $p_{I \mapsto I}(x), p_{S \mapsto I}(x)$ similarly. With the above lemma, we have the following uniform lower bound for them when $x \ge \varepsilon$. 
	
	\begin{corollary}\label{cor:ProbOfInfectionGoodRound}
		Given $\varepsilon > 0$, there exists a constant $C = C(\varepsilon, \alpha) > 0$ such that the following hold: a.s., for $x \ge \varepsilon$, 
		$$p_{R \mapsto I}(x), p_{I \mapsto I}(x), p_{S \mapsto I}(x) \ge C\wla. $$
	\end{corollary}
	\begin{proof}
		Wlog, suppose $\varepsilon < 1$ and $x \in [\varepsilon, 1]$ as for larger $x$, we can simply condition on what happens during the first $x-1$ units of time. A forced infection from $\rho$ to $v$ happens if
		\begin{itemize}
			\item[(i)] There is an event from $\frak{D}_v$ within $[0, \varepsilon/2)$, which happens with probability $(1 - e^{-\alpha\varepsilon/2})$,
			\item[(ii)] There is no event from $\frak{R}_v$ within $[0, x]$, with probability $e^{-x} \ge e^{-1}$,
			\item[(iii)] There is an event from $\frak{I}_{\rho v}$ within $[\varepsilon/2, \varepsilon]$, with probability $(1 - e^{-\la \varepsilon/2}) \ge (1 - e^{-\varepsilon/2})\wla$.
		\end{itemize}
		Together, the probability of $v$ being forced-infected and remaining infected until time $x$ is at least 
		$$(1 - e^{-\alpha\varepsilon/2})e^{-1} (1 - e^{-\varepsilon/2}) \wla =: C(\alpha, \varepsilon) \wla,$$
		as desired. 
	\end{proof}

    We now prove the following lemma, which provides a lower-bound on the number of rounds and $\varepsilon$-good rounds for the forced-infection process on a standalone star $N(\rho)$. Slightly abusing notations, to avoid introducing additional notations, we use the same $\tau_i$ notation for the process on $\mathcal T$ for the standalone star in the lemma below.
    \begin{lemma}\label{lm:probOfSurvivalSRoundsStar}
        Consider the standalone star graph $\rho$ with $n$ leaves $v_1, \dots, v_n$. Suppose $0 = \tau_1 < \tau_2 < \dots$ are the beginning of the rounds of $\rho$ (so we start out with $\rho$ being infected). Then there exists $\varepsilon_0(\alpha) > 0$ such that for all fixed $\varepsilon_0 \in (0, \varepsilon_0(\alpha)]$, there exist $c, C_0, C_1, \varepsilon > 0$ only dependent on $\alpha, \varepsilon_0$ such that the following hold: let $1 \le m_1 < m_2 \le \dots$ be the (random) indices of all $\varepsilon$-good rounds. For $\la > 0$, sufficiently large $n \ge C_0 \cdot \wla^{-2}$, and $2 \le T \le c(\wla^2 n)^{\alpha(1 - \varepsilon_0)}$, we have 
		$$ \P\left(\tau_{T} = \infty \mid m_1 = 1 \right) \le \P\left(\tau_{m_T} = \infty \mid m_1 = 1 \right) \le C_1\dfrac{T}{(\wla^2 n)^{\alpha(1 - \varepsilon_0)}} \le C_1c<1.$$
    \end{lemma}

    Due to monotonicity of the forced-infection process, we obtain the following immediate corollary.
	\begin{corollary}\label{cor:probOfSurvivalSRounds}
		Consider the graph $G$ containing $\rho$ as a vertex with degree at least $n$. Suppose $0 = \tau_1 < \tau_2 < \dots$ are the beginning of the rounds of $\rho$ (so we start out with $\rho$ being infected). For all fixed $\varepsilon_0 \in (0,\varepsilon_0(\alpha)]$, all the conclusions of Lemma \ref{lm:probOfSurvivalSRoundsStar} still holds for $N(\rho)$.
	\end{corollary}

    We first discuss the idea of the proof of Lemma \ref{lm:probOfSurvivalSRoundsStar} with a reformulation. For a standalone star, let $|\mathcal I_i^R|$ be the number of infected leaves in the forced-infection process when the root is recovered at the $i$-th round. The root $\rho$ switches back and forth between the states \textit{open} and \textit{closed}, as in the proof of Lemma \ref{lm:probInfectFarWithinShort}. Starting from $\tau_i^R$, condition on $|\mathcal I_i^R| \ge 2C\wla n$, we prove that w.h.p. within $K_0 := (1-\varepsilon') \log(\wla^2 n)$ units of time, there are $N_0 := C\wla^{-1} (\wla^2 n)^{\varepsilon'}$ leaves in $\mathcal I_i^R$ that remains infected all throughout, where $\varepsilon' \in (0, 1)$ is chosen appropriately later. Condition on this, we have $N_0$ infected leaves attempting to send infections to $\rho$. 
    
    We next define $A_t$ as follows.
    $$ A_t := \int_0^t 1\{\text{the root is open throughout } [\tau_i^R + s, \tau_i^R + s+ \varepsilon]\}ds. $$
    In other words, $A_t$ is the amount of time for which any incoming infection (which is necessarily a forced-infection by definition) guarantees an $\varepsilon$-good round. Condition on (i) $A_t$, and (ii) having $N_0$ infected leaves all throughout $K_0$ units of time, the probability that there are no $\varepsilon$-good rounds initiated within $K_0$ units of time is the probability that the Poisson clock $\PPP(\la N_0)$ does not have any event in $A_t$ units of time, which is $\exp(-\la N_0 A_{K_0})$. Thus, the probability of failing to have an $\varepsilon$-good round after $\tau_i^R$ is at most $$\E\exp(-\la N_0 A_{K_0}),$$
    and we are to upper bound this quantity.

    \begin{proof}[Proof of Lemma \ref{lm:probOfSurvivalSRoundsStar}]
        The first inequality is straightforward. We will prove the second inequality by induction on $T$. First, note that 
        $$\P\left( \tau_{m_T} = \infty \mid m_1 = 1 \right)
            \le \P\left( \tau_{m_{T-1}} = \infty \mid m_1 = 1 \right) + \P\left( \tau_{m_T} = \infty \mid m_{T-1} < \infty \right),$$
        so the problem is reduced to showing that 
        $$\P\left( \tau_{m_T} = \infty \mid m_{T-1} < \infty \right) \le C_1(\wla^2 n)^{-\alpha(1-\varepsilon_0)}.$$
        We now perform further reductions. By definition, since the $m_{T-1}$-th round is $\varepsilon$-good, by Corollary \ref{cor:ProbOfInfectionGoodRound}, the number of infected leaves stochastically dominates $\text{Bin}\left(n, 4C \wla\right)$ for some $C = C(\alpha,\varepsilon) > 0$, so by a Chernoff bound, with exponentially high probability, we have $|\mathcal I_{\tau_{m_{T-1}}}^R| \ge 2C\wla n$. Condition on this, the number of leaves in $I_{\tau_{m_{T-1}}}^R$ staying infected all throughout $K_0$ units of time is $\sim \text{Bin}\left(2C\wla n, e^{-K_0}\right)$. Similar to the proof of \cite[Equation (5.3)]{lam2024optimal}, with probability $1 - O_{\alpha, \varepsilon, \varepsilon'}\left( (\wla^2 n)^{-10\alpha} \right)$, there are at least $N_0 := C\wla^{-1}(\wla^2 n)^{\varepsilon'}$ such leaves. Subsequently, since we will implicitly condition on $\mathcal F_{\tau_{m_{T-1}}^R}$ and on $\{\tau_{m_{T-1}}^R < \infty\}$, for clarity of presentation, we will abuse notation, set $\tau_{m_{T-1}}^R$ to be our initial point, and omit it altogether. From here on out, our setting is reduced to the following.
        \begin{itemize}
            \item Initially, the root is closed, and we have at least $2C\wla n$ infected leaves.
            \item W.h.p. (as above) we have $N_0$ leaves that remains infected all throughout $[0, K_0]$.
            \item The root switches from open to closed at rate $1$, and from closed to open at rate $\alpha$. 
        \end{itemize}
        With the definition of $A_t$ as above (where $\tau_{m_{T-1}}^R$ is omitted), our task is reduced to choosing $\varepsilon, \varepsilon' \in (0, 1)$ suitably such that
        \begin{equation}\label{eq:SuboptBoundValidForcedInfections}
            \E\exp(-\la N_0 A_{K_0}) \lesssim_{\alpha, \varepsilon_0} \left( \wla^2 n \right)^{-\alpha(1-\varepsilon_0)}.
        \end{equation}

        Let $X_i \sim \Exp(\alpha)$ and $Y_i \sim \Exp(1)$ be the amount of time it takes for the root to switch from closed to open, and from open to closed in the $i$-th cycle, respectively. We also set $q_n := (\wla^2 n)^{-\varepsilon'/2} = o_n(1)$, and define $$\kappa = \inf\{t \ge 0: \rho \text{ is open on } [t, t + \varepsilon + q_n]\}.$$
        We observe that if $\kappa \le K_0 - q_n$, then $A_{K_0} \ge q_n$. Thus, 
        \begin{equation*}
            \E\exp(-\la N_0 A_{K_0}) \le \P\left( \kappa > K_0 - q_n \right) + \exp(-\la N_0 q_n),
        \end{equation*}
        and since $\exp(-\la N_0 q_n) \ll_{\alpha, \varepsilon,\varepsilon',\varepsilon_0} (\wla^2 n)^{-\alpha(1-\varepsilon_0)}$, we are left to choose $\varepsilon, \varepsilon' \in (0,1)$ such that
        \begin{equation}\label{eq:tailBoundForKappaStar}
            \P\left( \kappa > K_0 - q_n \right) \lesssim_{\alpha, \varepsilon_0} (\wla^2 n)^{-\alpha(1- \varepsilon_0)}.
        \end{equation}
        Let $G$ be the first time for which $Y_i \ge \varepsilon + q_n$, i.e. $\max_{i < G} Y_i < \varepsilon + q_n \le Y_G$. This means $G$ is geometrically distributed with success probability $\P(\Exp(1) \ge \varepsilon + q_n) = e^{-(\varepsilon + q_n)}$. Now, $\kappa > K_0 - q_n$ implies that 
        $$K_0 - q_n \le \sum_{i=1}^G X_i + \sum_{i=1}^{G-1}Y_i < \sum_{i=1}^G X_i  + (G-1) (\varepsilon + q_n).$$
        We now choose a conservative bound $t_* > 0$ for which
        \begin{equation}\label{eq:constraintsFortStarMGF}
        t_* < 1, \qquad \alpha t_*(\varepsilon + q_n) + \log\dfrac{1}{1 - t_*} < - \log(1 - e^{-(\varepsilon + q_n)}).   
        \end{equation}
        With the explicit choice $t_* = 1 - \varepsilon^{2/3} < 1$, then for sufficiently small $\varepsilon \le \varepsilon(\alpha)$, we have
        $$\alpha\varepsilon(1 - \varepsilon^{2/3}) + \log \dfrac{1}{\varepsilon^{2/3}} < -\log(1 - e^{-\varepsilon}),$$
        so perturbations by $q_n = o_n(1)$ does not change \eqref{eq:constraintsFortStarMGF} for sufficiently large $n$. We now have
        \begin{align*}
            \P\left( \kappa > K_0 - q_n \right) &\le \P\left(K_0 - q_n < \sum_{i=1}^G X_i  + (G-1) (\varepsilon + q_n), G \sim \text{Geom}(e^{-(\varepsilon + q_n)}) \right), \\
        \intertext{which, by using moment generation functions together with Markov inequality, is at most}
            &\le \inf_{0 < t < \alpha t_*}  e^{-t(K_0 - q_n)}\E\exp\left( t \sum_{i=1}^G X_i + t(G - 1)(\varepsilon + q_n) \right) \\
            &\approx_{\varepsilon} \inf_{0 < t < \alpha t_*}  e^{-tK_0} \E \left( \left( \dfrac{\alpha}{\alpha - t}\right)^G \cdot \exp( t(\varepsilon + q_n) G) \right) \\
            &\lesssim_{\varepsilon} \inf_{0 < t < \alpha t_*} \exp\left( -tK_0 + t(\varepsilon + q_n) + \log\dfrac{\alpha}{\alpha - t} \right) \\
            &= \inf_{0 < t < t_*} \exp\left( -\alpha tK_0 + \alpha t(\varepsilon + q_n) + \log\dfrac{1}{1 - t} \right), \\
        \intertext{and with $t = 1 - \varepsilon^{1/2} < t_*$, the above is at most}
            &\lesssim_{\varepsilon} \exp\left( -\alpha(1 - \varepsilon^{1/2})(1 - \varepsilon')\log(\wla^2 n) \right) = (\wla^2 n)^{-\alpha(1 - \varepsilon^{1/2})(1-\varepsilon')} \le (\wla^2 n)^{-\alpha(1 - \varepsilon_0)},
        \end{align*}
    for $\varepsilon' = \varepsilon^{1/2}$ and $\varepsilon \le \left(1 - \sqrt{1 - \varepsilon_0}\right)^2$. This proves \eqref{eq:tailBoundForKappaStar}, completing our proof.
    \end{proof}
	
	Conditioned on the process reaching round $T$, we next show that w.h.p. the time between each round is at most $(4\alpha + 3)\log n$. Just as how Corollary \ref{cor:probOfSurvivalSRounds} followed from Lemma \ref{lm:probOfSurvivalSRoundsStar}, we will prove the following lemma in the setting of an $n$-leaf star.
	
	\begin{lemma}\label{lm:tailBoundForTauS}
		Consider the standalone star graph $\rho$ with $n$ leaves $v_1, \dots, v_n$. Suppose $0 = \tau_1 < \tau_2 < \dots$ be the times at which the root $\rho$ gets infected. Then for $\alpha, \varepsilon_0$ as in the setting of Lemma \ref{lm:probOfSurvivalSRoundsStar}, such that for all $\la > 0$, and $n, T$ in the setting of Corollary \ref{cor:probOfSurvivalSRounds}, we have, for all $i \le T$,
		$$ \P(\tau_{i+1} - \tau_i > (4\alpha + 3) \log n \mid \tau_{i+1}< \infty, \tau_1^R > \varepsilon) \le n^{-2\alpha}$$

        As a result, due to monotonicity, the same result holds for any embedded star graph.
	\end{lemma}
	
	\begin{proof}[Proof of Lemma \ref{lm:tailBoundForTauS}]
	First, note that $(\tau_{i+1} - \tau_i^R)1_{\{\tau_{i+1} < \infty \}}$ is bounded by the maximum recovery time of all infected vertices at time $\tau_i^R$, which is stochastically dominated from above by a distribution $M_n$, where $M_n$ is the maximum of $n$ i.i.d. $\sim \Exp(1)$ random variables. Thus, since
	\begin{align*}
		(\tau_{i+1} - \tau_i) 1_{\{\tau_{i+1}<\infty\}} &\le \left( \tau_i^R - \tau_i \right)1_{\{\tau_i < \infty\}} + \left( \tau_{i+1} - \tau_i^R \right)1_{\{\tau_{i+1} < \infty \}},
	\end{align*}
	by a standard coupling argument, the sum above is stochastically dominated from above by the sum of an $\Exp(1)$ random variable and a random variable that is $\sim M_n$, the two of which are independent. Thus, together with the fact that $$\P\left( \tau_{i+1} < \infty \mid \tau_1^R > \varepsilon \right)  \ge \P\left( \tau_{T+1} < \infty \mid \tau_1^R > \varepsilon \right) \ge 1 - C_1 c = \Theta_{\alpha, \varepsilon_0}(1),$$ that follows from Lemma \ref{lm:probOfSurvivalSRoundsStar}, we have that
	\begin{align*}
		&\P\left(\tau_{i+1} - \tau_i > (4\alpha + 3) \log n \mid \tau_{i+1}< \infty, \tau_1^R > \varepsilon \right) \\
		\lesssim_{\alpha, \varepsilon_0}&\ \P\left( (\tau_{i+1} - \tau_i )1_{\{\tau_{i+1}<\infty\}} > (4 \alpha + 3)\log n \right) \\
		\le&\ \P\left( \Exp(1) > (2\alpha + 1) \log n \right) + \P\left( M_n > (2\alpha +2) \log n \right) \\
		\le&\ n^{-(2\alpha + 1)} + \left( 1 - \left(1 - n^{-(2\alpha+ 2)} \right)^n \right),
		\intertext{which, by Bernoulli's inequality, is at most} 
		\le&\ n^{-(2\alpha + 1)} + n \cdot n^{-(2\alpha + 2)} \ll n^{-2\alpha},
	\end{align*}
	completing our proof.
	\end{proof} 

    \subsection{Towards the proof of Theorem \ref{thm:mainThmOfStrongSurvival}}
	We first introduce the following parameter choices that will be used throughout the remainder of this subsection. First, we fix $T_1 := \frac{c}{10^5}(\wla^2 n)^{\alpha(1-\varepsilon_0)}$, where $\varepsilon_0$ is sufficiently small depending on $\mu$ and $\alpha$; $c, \varepsilon$ are constants chosen appropriately according to Lemma \ref{lm:probOfSurvivalSRoundsStar} corresponding to $\varepsilon_0$; $n$ is chosen appropriately large later. Denote $T_2 := (\wla^2 n)^{\alpha\varepsilon_0} \ll T_1$. Let $r_0 := (4\alpha + 3) \log n$ and $r$ is an integer such that $\left( \frac{\wal \wla}{1 + (1 - \wal)(1 - \wla)} \right)^{r} T_2 \le 1$. We will pick $r = \Theta_{ \mu, \la, \varepsilon_0, \alpha}(-\log \mu([n, \infty)) + \log n) \gtrsim r_0$ appropriately later; for now, we will keep the notation as-is for clarity. The only (soft) requirements for now are that $T_1 \gg T_2 r$ and that $r \gg 1$. Finally, we denote $r' := r_0 + 3r$ and $T := T_2 r'$, where $T$ is the time scale of our look-back window. For clarity of presentation, we treat $T_1, T_2, r_0, r$ as integers. We will fix our choices of $\varepsilon_0, \varepsilon, T_1, T_2, T, r_0, r, r'$ as above, and make them explicit later on.
	
	With the notations above, we are ready to state and prove the first lemma, which lower bounds the probability that a vertex that is of distance $r$ away from the root will be infected by time $T$. 
	\begin{lemma}\label{lm:probInfectFarWithinShortRepeated}
		Suppose $0 = \tau_1 < \tau_2 < \dots $ be the times at which the root $\rho$ gets infected. Moreover, let $\rho = w_0, w_1, w_2, \dots, w_r$ be a path within our graph $\T(\rho, \mu)$ with edges $\{w_i, w_{i+1}\}$, $0 \le i < r$. Then
		$$\P\left( w_r \text{ is infected within time } [0, T] \right) \ge \dfrac{e^{-\varepsilon}}{4}\left( 1 - \dfrac{2}{10^5}\right)\left( \dfrac{\wal \wla}{1 + (1 - \wal)(1 - \wla)} \right)^r T_2.$$
	\end{lemma} 
	\begin{proof}
		Note that $e^{-\varepsilon}$ is the probability that the first round is $\varepsilon$-good, so we can simply condition on this event moving forward, and simply denote it "good". We will work on the event that the root is re-infected for at least $T_1 \gg T_2r' = T$ times, allowing for an abundance of root re-infection. We define the following events, whose use is made clearer throughout the proof of the lemma.
		\begin{align*}
			\mathcal A &:= \{\tau_{T_1} > T\} \cap \{\exists \{\tilde \tau_i\}_{0 \le i < T_2}: \tilde \tau_i \in [i r', i r' + r_0], \rho \text{ is infected at time } \tilde \tau_i \}, \\
			\mathcal B &:= \{\tau_{T_1} < \infty, N(\rho) \text{ is infected throughout } [0, \tau_{T_1}]\}.
		\end{align*}
		
		In words, $\mathcal B$ is simply the probability of surviving at least $T_1$ rounds. $\mathcal A$ is the event that the $T_1$ rounds of infection covers $[0, T]$, and that the root is infected within $T_2$ specified intervals of length $r_0$, after which we will look at the $3r$-time interval afterwards to see if $w_r$ is successfully infected. 
		
		By Lemma \ref{lm:probOfSurvivalSRoundsStar}, we readily have that $\P(\mathcal B \mid \text{good}) > 1 - 10^{-5}$. On $\mathcal B$, we observe that $\P( \tau_{T_1} \le T) = o_n(1)$. This follows from the fact that $\tau_{T_1}$ stochastically dominates the sum of $(T_1 - 1)$ i.i.d. $\Exp(1)$ clocks, and that $T_1 \gg T_2 r' = T$. We can then apply Paley-Zygmund's inequality appropriately. 
		
		On $\mathcal B \cap \{\tau_{T_1} > T\}$, we divide the time interval $[0, T] = [0, T_2r']$ into $T_2$ intervals $\{[i r', (i+1) r']: 0 \le i < T_2\}$. By Lemma \ref{lm:tailBoundForTauS} and a union bound, with probability $1 - T_1 n^{-2\alpha} = 1 - o_n(1)$, all of the first successful $T_1$ rounds are of length at most $r_0$. That means that for every time interval of length $r_0$, we can find a time for which the root is infected. In other words, we can find (random) times $\{\tilde \tau_i\}_{0 \le i < T_2}$ such that the root $\rho$ is infected at time $\tilde \tau_i \in [i r', i r' + (4\alpha + 3) \log n] = [i r', i r' + r_0]$. Thus, for $n$ sufficiently large,
		$$ P(\mathcal A, \mathcal B \mid \text{good}) = \P(\mathcal B \mid \text{good}) \pm o_n(1) > 1 - \dfrac{2}{10^5}, $$ 
		
		Finally, we see that conditioned on $\mathcal A, \mathcal B$, by Lemma \ref{lm:probInfectFarWithinShort}, the probability that at least one of the $T_2$ attempts to send an infection from $\rho$ to $w_r$ within the intervals $\{[\tilde \tau_i, \tilde \tau_i + 3r]: 0 \le i < T_2\}$ is at least
		$$1  - \left(1 - \dfrac{1}{2}\left( \dfrac{\wal \wla}{1 + (1 - \wal)(1 - \wla)} \right)^r \right)^{T_2}\ge \dfrac{1}{4}\left( \dfrac{\wal \wla}{1 + (1 - \wal)(1 - \wla)} \right)^r T_2,$$
		where independence applies since $\{[\tilde \tau_i, \tilde \tau_i + 3r]: 0 \le i < T_2\}$ are non-overlapping time intervals. From the above,
		\begin{align*}
			&\P\left( w_r \text{ is infected within time } [0, T] \right) \\
			\ge&\ e^{-\varepsilon}\Pr\left(\mathcal A, \mathcal B \mid \text{good} \right) \P\left( w_r \text{ is infected within time } [0, T]\ \bigg\vert \ \mathcal A, \mathcal B, \text{good} \right) \\
			\ge&\ \dfrac{e^{-\varepsilon}}{4}\left( 1 - \dfrac{2}{10^5}\right)\left( \dfrac{\wal \wla}{1 + (1 - \wal)(1 - \wla)} \right)^r T_2,
		\end{align*}
		completing our proof.
	\end{proof}
	
	Let us proceed with the main proof. In the tree $\T(\rho, \mu)$, let $J$ be the number of vertices among the $r$-th generation vertices that has at least $n$ children in $\mathcal T$, and denote these vertices $v_1, \dots, v_{J}$. Let $M$ be the number of vertices among those $J$ that are infected by time $T$ in the forced-infection process on this tree.
	
	Showing that $\inf_{s \le 3T} \phi(s) > 0$ is rather straightforward: we only have to deal with the case $2T < s \le 3T$, which can be dealt with using ideas from Lemma \ref{lm:probInfectFarWithinShortRepeated}. Thus, we now consider $t > 3T$. Condition on $M$, the number $N(t)$ of vertices among them that are (i) infected within time $[t - 3T, t - T]$, and (ii) its last infection interval before time $(t - T)$ is of length at least $\varepsilon$ stochastically dominates
	$$ \text{Bin}(M, e^{-\varepsilon} \cdot \inf_{0 \le s \le t - T} \phi(s)). $$
	The parameter $\inf_{0 \le s \le t - T} \phi(s)$ is straightforward by the definition of $\phi(s)$. The binomial form of the lower bound needs to be dealt with more carefully - this is exactly where all the pain of defining forced-infection culminated in. When we restrict to the forced-infection processes to the Galton-Watson trees rooted at each $v_i$, we see that infections (i) only depend on the first time that $v_i$ is infected, and (ii) is measurable with respect to $\mathcal F((GW(\mu), v_i))$, where $(GW(\mu), v_i)$ is the Galton-Watson tree rooted at $v_i$. We see that $\bigg(\mathcal F((GW(\mu), v_i))\bigg)_{i \ge 1}$ are independent of each other, hence the forced-infection processes are also independent of each other.
	
	At this point, our goal is to apply Lemma \ref{lm:PemLem23}. There is a subtle point: it only applies to when $M$ is positive integer-valued, which is not readily available to us. To this end, we denote $\tilde M := M \mid_{M > 0}$, i.e. $M$ conditioned on the event that $\{M > 0\}$. By Lemma \ref{lm:PemLem23} applied to $\tilde M$, there exists constants $\delta', \varepsilon' > 0$ (we can choose $\delta'$ freely; $\varepsilon'$ will then depend on our choice of $r, n, \delta'$) such that the probability that at least one of $\{v_i\}_{1 \le i \le J}$ that are infected within $[t - 3T, t - T]$ is at least
	\begin{align*}
		\P(N(t) \ge 1) &\ge (1 - \P(M = 0)) \left[\left((1 - \delta') (\E \tilde M) e^{-\varepsilon} \cdot \inf_{0 \le s \le t - 2T} \phi(s) \right) \wedge \varepsilon'\right] \\
		&= \left((1 - \delta') (\E M) e^{-\varepsilon} \cdot \inf_{0 \le s \le t - 2T} \phi(s) \right) \wedge \varepsilon',
	\end{align*}
	where the definition of $\varepsilon'$ is  changed between inequalities; the only thing that matters is that $\varepsilon' > 0$ all throughout. Condition on $\{N(t) \ge 1\}$, let $v_I$ be a (random) vertex among the aforementioned $N(t)$ vertices, then the (conditional) probability that $v_I$ re-infects the root $\rho$ within time $[t - 3T, t]$ can be bounded from below using Lemma \ref{lm:probInfectFarWithinShortRepeated}. With all the above, the probability that $\rho$ is infected within time $[t - 3T, t]$ is at least (we apply Lemma \ref{lm:probInfectFarWithinShortRepeated} twice, and again abuse notations for $\varepsilon'$)
	\begin{align}
		&\P\left( \exists s \in [t - 3T, t]: \rho \in I_s\right) \nonumber\\
		\ge&\ \left(\dfrac{e^{-\varepsilon}}{4}\left( 1 - \dfrac{2}{10^5}\right)\left( \dfrac{\wal \wla}{1 + (1 - \wal)(1 - \wla)} \right)^r T_2\right) \left(\left((1 - \delta') (\E M) e^{-\varepsilon} \cdot \inf_{0 \le s \le t - 2T} \phi(s) \right) \wedge \varepsilon' \right) \nonumber \\
		\ge&\ \left(\dfrac{e^{-\varepsilon}}{4}\left( 1 - \dfrac{2}{10^5}\right)\left( \dfrac{\wal \wla}{1 + (1 - \wal)(1 - \wla)} \right)^r T_2\right)^2 \left(\left((1 - \delta') (\E J) e^{-\varepsilon} \cdot \inf_{0 \le s \le t - 2T} \phi(s) \right) \wedge \varepsilon' \right), \label{eq:LowBoundProbReinfectRoot1}
	\end{align}
	We are now left to lower bound the probability that $\rho$ is infected within time $[t - 2T, t]$ conditioned on it being infected within time $[t - 3T, t]$. We have the following lemma.
	
	\begin{lemma}\label{lm:constantProbHelperRoot}
		We have that
		$$ \P\left( \exists s \in [t - 2T, t]: \rho \in I_s \mid \exists s \in [t - 3T, t]: \rho \in I_s \right) \ge e^{-\varepsilon}\left(1 - \dfrac{2}{10^5}\right).$$ 
	\end{lemma} 
	
	\begin{proof}
        Once again, due to monotonicity, we can reduce to the setting of a standalone star. Moreover, the $e^{-\varepsilon}$ term again comes from conditioning on the first $\varepsilon$-good round of the root, so for the rest of the proof, we condition on this and omit the term $e^{-\varepsilon}$. The problem reduces to showing that on a standalone star with $n$ leaves,
		\begin{equation*}
			\inf_{u \in [0, T]} \P\left( \exists s \in [u , u + 2T]: \rho \in I_s \mid \rho \in I_0, \tau_1^R > \varepsilon \right) \ge 1 - \dfrac{2}{10^5}.
		\end{equation*}
		
		Fix any $u \in [0, T]$. Let $\Psi$ be the number of rounds that the root $\rho$ is successfully infected, we have
		\begin{align}
			&\P\left( \nexists s \in [u, u+ 2T]: \rho \in I_s \mid \rho \in I_0, \tau_1^R > \varepsilon \right) \nonumber \\
			\le&\ \P\left( \nexists s \in [u, u+2T]: \rho \in I_s, \Psi \ge T_1 \mid \rho \in I_0, \tau_1^R > \varepsilon \right) + \P\left(\Psi < T_1 \mid \rho \in I_0, \tau_1^R > \varepsilon \right), \nonumber 
			\intertext{and applying Corollary \ref{cor:probOfSurvivalSRounds} for the second term, the above is}
			<&\ \P\left( \nexists s \in [u, u+ 2T]: \rho \in I_s, \Psi \ge T_1 \mid \rho \in I_0, \tau_1^R > \varepsilon \right) + \dfrac{1}{10^5}. \label{eq:constantProbHelperRoot1}
		\end{align}
		
		Now, the event $\{\nexists s \in [u, u+ 2T]: \rho \in I_s, \Psi \ge T_1\}$ implies that for the $T_1$ successful rounds, either all of them are within $[0, u]$, or there is a round $i < T_1$ for which $\tau_i^R < u < u + 2T < \tau_{i+1} < \infty$, i.e. the interval $[u,  u + 2T]$ is well-contained between two forced-infection intervals. The former is of probability $o_n(1)$, as reasoned in Lemma \ref{lm:probInfectFarWithinShortRepeated} using Paley-Zygmund inequality.
		In the second case, condition on $\mathcal F_{\tau_i^R}$, on the event $\{\tau_i^R < \infty\}$ and $\{\rho \in I_0, \tau_1^R > \varepsilon\}$, we see that $1_{\{\tau_{i +1} < \infty \}}(\tau_{i+1} - \tau_i^R)$ is dominated from above by the maximum of $n$ independent $\sim \Exp(1)$ random variables (recovery time of all leaves), so by Lemma \ref{lm:maxOfExpRVs},
		$$ \E\left( 1_{\{\tau_{i +1} < \infty \}}(\tau_{i +1} - \tau_i^R)^K \ \bigg\vert \ \mathcal F_{\tau_i^R} \right) \lesssim_{K} (\log n)^K1_{\{\tau_i < \infty\}}.$$ 
		 By Markov inequality and applying Law of Total Expectation suitably, fix $K \ge \frac{2 (1 - \varepsilon_0)}{\varepsilon_0}$, we have 
		\begin{align}
			&\P\left( \Psi \ge T_1, \exists i < T_1: \tau_{i + 1} - \tau_i^R > 2T \mid \rho \in I_0, \tau_1^R > \varepsilon \right) \nonumber \\
			\le&\ \P\left( \exists i < T_1: 2T < \tau_{i + 1} - \tau_i^R < \infty \mid \rho \in I_0, \tau_1^R > \varepsilon \right) \nonumber \\
			\lesssim_{\alpha, \varepsilon_0}&\ T_1 \dfrac{(\log n)^K}{(2T)^K} \lesssim_{\alpha, \varepsilon_0} \dfrac{(\log n)^K}{(\wla^2 n)^{\alpha(1 - \varepsilon_0)}  }(r')^{-K} \ll_{\alpha, \varepsilon_0} 1.
		\end{align}
		Thus, the first term in \eqref{eq:constantProbHelperRoot1} is $o_n(1)$, giving the desired result.
	\end{proof}
	
	From equation \eqref{eq:LowBoundProbReinfectRoot1}, Lemma \ref{lm:constantProbHelperRoot}, and the fact that $\E J \ge n d^{r-1}\mu([n, \infty))$, by using Pemantle's strategy, it follows that strong survival happens whenever 
	\begin{equation}\label{eq:critStrongSurvivalEq}
		C_{\alpha, \la, \varepsilon_0} \left( \dfrac{\wal \wla \sqrt{d}}{1 + (1 - \wal)(1 - \wla)} \right)^{2(r-1)} (\wla^2 n)^{2\alpha\varepsilon_0} n\mu([n, \infty)) > 1,
	\end{equation}
	for some $C_{\alpha,\la, \varepsilon_0} \in (0, 1)$ only dependent on $\alpha, \la, \varepsilon_0$ and $n$ satisfying the constraints in Corollary \ref{cor:probOfSurvivalSRounds}. 
	
	\begin{proof}[Proof of Theorem \ref{thm:mainThmOfStrongSurvival}]
		It remains to show that it is possible to choose $n$ and $r$ satisfying the constraints of Corollary \ref{cor:probOfSurvivalSRounds}, Lemma \ref{lm:probInfectFarWithinShortRepeated}, equation \eqref{eq:critStrongSurvivalEq}, and such that $\left( \frac{\wal \wla}{1 + (1 - \wal)(1 - \wla)} \right)^rT_2 \le 1, r \ll T_1 / T_2 = \Theta_{\alpha, \varepsilon_0} ((\wla^2 n)^{\alpha(1- 2\varepsilon_0)})$. For $\la > \frac{\alpha + 2}{\alpha + 1} \cdot \left( \wal \sqrt{d} - 1 \right)^{-1}$, \eqref{eq:critStrongSurvivalEq} becomes finding $n, r$ such that the following conditions hold.
		\begin{align*}
			1 &< C_{\alpha, \la, \varepsilon_0} \left( \dfrac{\wal \wla \sqrt{d}}{1 + (1 - \wal)(1 - \wla)} \right)^{r-1} n^{2\alpha\varepsilon_0 + 1}\mu([n, \infty)), \\
			r &\ge \dfrac{\alpha \varepsilon_0 \log (\wla^2 n)}{\log \wal^{-1} + \log \wla^{-1} + \log(1 + (1 - \wal)(1 - \wla))}, \\
			r &\ll_{\alpha, \la, \varepsilon_0} \ n^{\alpha(1- 2\varepsilon_0)},
		\end{align*}
		where $C_{\alpha, \la, \varepsilon_0} \in (0, 1)$ is dependent on $\alpha, \la, \varepsilon_0$. Since $\frac{\wal \wla \sqrt{d}}{1 + (1 - \wal)(1 - \wla)} > 1,$ the first and third condition above implies that we are to choose $n, r$ such that 
		\begin{align*}
			n^{\alpha(1 - 2\varepsilon_0)} \gg_{\alpha, \la, \varepsilon_0} r \ge - \dfrac{\log \mu([n, \infty))  + \log C_{\alpha,\la, \varepsilon_0}  + (2\alpha\varepsilon_0 + 1) \log n}{\log (\wla \wal \sqrt{d}) - \log( 1 + (1 - \wal)(1 - \wla))} + 1.
		\end{align*}
		We can pick $\varepsilon_0$ sufficiently small (depending on $\mu$), then $r$ and $n$ that satisfy the above due to the constraint $\limsup \mu([n, \infty)) \exp(n^{\alpha - \delta}) = \infty$ for all $\delta > 0$ sufficiently small. The second condition can be incorporated without any contradictions, completing our proof.
	\end{proof}

	\bibliographystyle{alpha}
	\bibliography{refs}
	
\end{document}